\documentclass[final,leqno]{siamltex}
\usepackage{graphicx}
\usepackage{algorithm} 
\usepackage{algorithmic} 

\title{On Fast Implementation of Higher Order Hermite-Fej\'{e}r Interpolation\thanks{This work was
supported by National Science Foundation of China (No. 11371376).}}

\author{Shuhuang Xiang$^{\dag}$ and Guo He$^\ddag$\thanks{Department of Applied
Mathematics and Software, Central South University, Changsha, Hunan
410083, P. R. China.}$^{\ddag}$\footnotesize{Corresponding author}}

\begin{document}
\maketitle

\begin{abstract}
The problem of barycentric Hermite interpolation is highly susceptible to overflows or
underflows.  In this paper, based on Sturm-Liouville equations for Jacobi orthogonal polynomials, we consider the fast implementation on the second barycentric  formula for higher order Hermite-Fej\'{e}r interpolation at Gauss-Jacobi or Jacobi-Gauss-Lobatto pointsystems, where the barycentric weights can be efficiently evaluated and cost linear operations corresponding to the number of grids totally. Furthermore, due to the division of the second barycentric form, the exponentially increasing common factor in the barycentric weights can be canceled, which yields a superiorly stable method for computing the simplified barycentric weights, and leads to a fast implementation of the higher order Hermite-Fej\'{e}r interpolation with linear operations on the number of grids. In addition, the convergence rates are derived for Hermite-Fej\'{e}r interpolation at Gauss-Jacobi pointsystems.
\end{abstract}



\pagestyle{myheadings}
\thispagestyle{plain}

\section{Introduction}
There are many investigations for the behavior of continuous functions approximated by polynomials. Weierstrass \cite{Weierstrass} in 1885 proved the well known result that every continuous function $f(x)$ in $[-1,1]$ can be uniformly approximated as closely as desired by a polynomial function. This result has both practical and theoretical relevance, especially in polynomial interpolation.

Polynomial interpolation is a fundamental tool in many areas of numerical analysis. Lagrange interpolation is a well known, classical technique for approximation of continuous functions. Let us denote by
\begin{equation}\label{fundamentalpoints}
x_1^{(n)},\,\,\, x_2^{(n)},\,\,\, \ldots,\,\,\, x_n^{(n)}
\end{equation}
the $n$ distinct points in the interval $[-1,1]$ and let $f(x)$ be a function
defined in the same interval. The $n$th Lagrange interpolation
polynomial of $f(x)$ is uniquely  defined by the formula
\begin{equation}\label{lagrange}
L_n[f] =\sum_{k=1}^nf(x_k^{(n)})\ell_k^{(n)}(x),\quad \ell_k^{(n)}(x)=\frac{\omega_n(x)}{\omega'_n(x_k^{(n)})(x-x_k^{(n)})},
\end{equation}
where $\omega_n(x)=(x-x_1^{(n)})(x-x_2^{(n)})\cdots(x-x_n^{(n)})$. However, for an arbitrarily given
system of points $\{x_1^{(n)},x_2^{(n)},\ldots,x_n^{(n)}\}_{n=1}^{\infty}$, Bernstein \cite{Bernstein1914} and Faber \cite{Faber}, in 1914, respectively, showed that there exists a  continuous function $f(x)$ in $[-1,1]$ for which the sequence $L_n[f]$ ($n=1,2,\ldots$) is not uniformly convergent to $f$ in $[-1,1]$\footnote{A very simple proof was given by Fej\'{e}r \cite{Fejer1930} in 1930.} . Additionally, Bernstein \cite{Bernstein1931} proved that there exists a continuous function $f(x)$ also for which the sequence $L_n[f]$ is divergent. Particularly, Gr\"{u}nwald \cite{Grunwald1935} in 1935 and Marcinkiewicz \cite{Marcinkiewicz} in 1937, independently, showed that even for the Chebyshev points of first kind
\begin{equation}
x_k^{(n)}=\cos\left(\frac{2k-1}{2n}\pi\right),\quad k=1,2,\ldots,n,\quad n=1,2,\ldots,
\end{equation}
there is a continuous  function $f(x)$ in $[-1,1]$ for which the sequence $L_n[f]$ is divergent everywhere in $[-1,1]$.

\subsection{(Higher order) Hermite-Fej\'{e}r interpolation}
One of the proofs of Weierstrass¡¯ approximation theorem using interpolation
polynomials was presented by Fej\'{e}r \cite{Fejer1916} in 1916  based on the above Chebyshev pointsystem (1.3): If $f\in C[-1,1]$, then there is a unique polynomial $H_{2n-1}(f,x) $ of degree at most
$2n - 1$ such that $\lim_{n\rightarrow \infty}\|H_{2n-1}(f)-f\|_{\infty}=0$, where $H_{2n-1}(f,x)$ is determined by
\begin{equation}
H_{2n-1}(f,x_k^{(n)})=f(x_k^{(n)}),\quad H_{2n-1}'(f,x_k^{(n)})=0,\quad k = 1, 2,\ldots, n.
\end{equation}
This polynomial is known as the Hermite-Fej\'{e}r interpolation polynomial.

The convergence result has been extended to general Hermite-Fej\'{e}r interpolation of $f(x)$  at nodes (1.1), upon strongly normal pointsystems introduced by Fej\'{e}r \cite{Fejer1932a}: Given, respectively,
the function values $f(x_1^{(n)})$, $f(x_2^{(n)})$, $\ldots$, $f(x_n^{(n)})$ and derivatives
$d_1^{(n)}$, $d_2^{(n)}$,$\ldots$, $d_n^{(n)}$ at these grids, the Hermite-Fej\'{e}r interpolation polynomial $ H_{2n-1}(f)$ has the form of \begin{equation}\label{hermite-fejer}
\quad \,\, H_{2n-1}(f,x) =\sum_{k=1}^nf(x_k^{(n)})h_k^{(n)}(x)+\sum_{k=1}^nd_k^{(n)}b_k^{(n)}(x),
\end{equation}
where $h_k^{(n)}(x)=v_k^{(n)}(x)\left(\ell_k^{(n)}(x)\right)^2$, $b_k^{(n)}(x)=(x-x_k^{(n)})\left(\ell_k^{(n)}(x)\right)^2$ and
$$
 v_k^{(n)}(x)=1-(x-x_k^{(n)})\frac{\omega_n''(x_k^{(n)})}{\omega_n'(x_k^{(n)})}\mbox{\quad (see Fej\'{e}r \cite{Fejer1932b}).}
$$

The pointsystem (1.1) is called  strongly normal if for all $n$
\begin{equation}\label{stronglynormal}
 v_k^{(n)}(x)\ge c>0,\quad k=1,2,\ldots,n,\quad x\in [-1,1]
\end{equation}
for some positive constant $c$. The pointsystem (1.1) is called  normal if for all $n$
\begin{equation}\label{normal}
 v_k^{(n)}(x)\ge 0,\quad k=1,2,\ldots,n,\quad x\in [-1,1].
\end{equation}
Fej\'{e}r \cite{Fejer1932a} (also see Szeg\"{o} \cite[pp 339]{Szego}) showed that for the zeros of Jacobi polynomial $P_n^{(\alpha,\beta)}(x)$ of degree $n$ ($\alpha>-1$, $\beta>-1$)
$$
{\small \quad \quad v_k^{(n)}(x)\ge \min\{-\alpha,-\beta\}\mbox{\quad for $-1<\alpha\le 0$, $-1<\beta\le 0$, $k=1,2,\ldots,n$ and $x\in [-1,1]$}.}
$$
While for the Legendre-Gauss-Lobatto pointsystem (the roots of {\small$(1-x^2)P_{n-2}^{(1,1)}(x)=0$}),
$$
 v_k^{(n)}(x)\ge 1,\quad k=1,2,\ldots,n,\quad x\in [-1,1].
$$
This result is extended to Jacobi-Gauss-Lobatto pointsystem (the roots of $(1-x^2)P_{n-2}^{(\alpha,\beta)}=0$) and  Jacobi-Gauss-Radau pointsystem (the roots of $(1-x)P_{n-1}^{(\alpha,\beta)}=0$ or $(1+x)P_{n-1}^{(\alpha,\beta)}=0$) by V\'{e}rtesi \cite{Vertesi1979a,Vertesi1979b}: for all $k$ and $x\in [-1,1]$,
$$
 v_k^{(n)}(x)\ge \min\{2-\alpha,2-\beta\}\mbox{\small\quad for $\{x_{k}^{(n-2)}\}\bigcup \{-1,1\}$ with  $1\le\alpha\le 2$ and $1\le\beta\le 2$,}
$$
$$
 v_k^{(n)}(x)\ge \min\{2-\alpha,-\beta\}\mbox{\quad for $\{x_{k}^{(n-1)}\}\bigcup \{1\}$ with  $1\le\alpha\le 2$ and $-1<\beta\le 0$,}
$$
$$
\quad\quad v_k^{(n)}(x)\ge \min\{-\alpha,2-\beta\}\mbox{\quad for $\{x_{k}^{(n-1)}\}\bigcup \{-1\}$ with  $-1<\alpha\le 0$ and $1\le\beta\le 2$}.
$$

Based upon the (strongly) normal pointsystem, Gr\"{u}nwald  \cite{Grunwald1942} in 1942 showed that for every $f\in C[-1,1]$, $\lim_{n\rightarrow \infty}\|H_{2n-1}(f)-f\|_{\infty}=0$ if $\{x_k^{(n)}\}$ is strongly normal satisfying (1.6) and $\{d_k^{(n)}\}$ satisfies
$$
|d_k^{(n)}|<n^{c-\delta} \mbox{\quad for some given positive number $\delta$},\quad k=1,2,\ldots,\quad n=1,2,\ldots,$$ while  $\lim_{n\rightarrow \infty}\|H_{2n-1}(f)-f\|_{\infty}=0$ in $[-1+\epsilon,1-\epsilon]$ for each fixed $0<\epsilon<1$ if  $\{x_k^{(n)}\}$ is  normal and $\{d_k^{(n)}\}$ is uniformly bounded for $n=1,2,\ldots$.

To get fast convergence on suitable smooth functions, higher order Hermite-Fej\'{e}r interpolation polynomials were considered in Goodenough and Mills \cite{GooMil1981}, Sharma and Tzimbalario \cite{ShaTzi1975}, Szabados \cite{Sza1993}, V\'{e}rtesi \cite{Vertesi1989}, etc.: for $p=0,1,\ldots,m-1$ and $q=1,2,\ldots,n$,
\begin{equation}
H_{mn-1}(f,x)=\sum_{k=1}^n\sum_{j=0}^{m-1}f^{(j)}(x_k^{(n)})A_{jk}(x), \quad H_{mn-1}^{(p)}(f,x_q^{(n)})=f^{(p)}(x_q^{(n)}),
\end{equation}
where the polynomial $A_{jk}(x)$ of degree at most $mn-1$ satisfies
\begin{equation}
  A_{jk}^{(p)}(x_q^{(n)})=\delta_{jp}\delta_{kq},\quad j=0,1,\ldots,m-1,\quad k=1,2,\ldots,n,
\end{equation}
and $\delta$ is the Kronecker delta function. For simplicity, in the following we abbreviate $x_k^{(n)}$ as $x_k$, $\ell_k^{(n)}(x)$ as $\ell_k(x)$, $h_k^{(n)}(x)$ as $h_k(x)$, and $b_k^{(n)}(x)$ as $b_k(x)$.

The convergences of the higher order Hermite-Fej\'{e}r interpolation polynomials  have been extensively studied (see e.g. Byrne et al. \cite{ByrMilSim1993}, Goodenough and Mills \cite{GooMil1981b}, Locher \cite{Locher}, Mathur and Saxena \cite{MatSax1967}, Moldovan \cite{Moldovan}, Nevai and V\'{e}rtesi \cite{Nevai}, Popoviciua \cite{Popoviciua}, \cite{ShaTzi1975}, Shi \cite{Shi1994,Shi2000}, Shisha et al. \cite{Shisha}, Sun \cite{Sun1983}, Szili \cite{Szi2008}, Vecchia et al. \cite{Della}, V\'{e}rtesi \cite{Vertesi1979a,Vertesi1983} etc.). The convergence rates are achieved most on  Gauss-Jacobi or Jacobi-Gauss-Lobatto pointsystems. As is well
known in approximation theory, the right approach is to use point sets that are clustered at the endpoints of the interval with an asymptotic density proportional to
$(1 - x^2)^{-1/2}$ as $n\rightarrow \infty$ emphasised, for example, in Berrut and Trefethen \cite{Berrut2004} and Trefethen \cite{Trefethenbook}. Hence, in this paper we confine ourselves to Gauss-Jacobi or Jacobi-Gauss-Lobatto pointsystems.


\subsection{Barycentric forms and implementation on general Hermite interpolation}
In general, the Hermite interpolation is to find a polynomial $H_{N-1}(f,x)$ of degree at most $N-1$ such that
$$
\frac{d^r}{dx^r}H_{N-1}(f,x)\bigg|_{x=x_k}=f_{k,r}\quad {\rm for}\quad r=0,1,\ldots,n_k-1,
$$
where $f_{k,0},f_{k,1},\ldots,f_{k,n_k-1}$ denote the function value and its first $n_k-1$ derivatives at the interpolation grid points $x_k$ ($k=1,2,\ldots,n$), respectively,
and $N=n_1+n_2+\cdots+n_n$.

The polynomial $H_{N-1}(f,x)$ can be represented in either the Newton form or the barycentric
form. In the Newton form, the grid points $x_k$ must be ordered  in a special way (see Schneider and Warner \cite{Schneider}).
If the grid points are not carefully ordered, the Newton form is susceptible to catastrophic
numerical instability. For more details, see Fischer and Reichel \cite{Fischer}, Tal-Ezer \cite{Tal}, Berrut and Trefethen \cite{Berrut2004}, Butcher et al. \cite{Butcher2011} and Sadiq and  Viswanath \cite{Sadiq2013}. In contrast, the barycentric form does not depend on the order
in which the nodes are arranged, which treats all the grid points equally. Barycentric interpolation is arguably the method of choice for numerical polynomial
interpolation.

The first barycentric formula for the Hermite interpolation is of the form of
\begin{equation}\label{bary form1}\begin{array}{lll}
  H_{N-1}(f,x) &=&H^{*}(f,x)\sum_{k=1}^n\frac{f_{k,n_k-1}}{(n_k-1)!}\left(\frac{w_{k,0}}{x-x_k}\right)+\frac{f_{k,n_k-2}}{(n_k-2)!}
  \left(\frac{w_{k,0}}{(x-x_k)^2}+\frac{w_{k,1}}{x-x_k}\right) \\
  &&+\cdots+f_{k,0}\left(\frac{w_{k,0}}{(x-x_k)^{n_k}}+\cdots+\frac{w_{k,n_k-1}}{x-x_k}\right),
\end{array}\end{equation}
where $H_N^{*}(f,x)=\prod_{k=1}^n(x-x_k)^{n_k}$ and $w_{k,r}$ is called the barycentric weights. Applying $1\equiv H_N^{*}(f,x)\sum_{k=1}^n\sum_{r=0}^{n_k-1}w_{k,r}(x-x_k)^{r-n_k}$ derives the second barycentric form\footnote{Two typos occur in (1.4) \cite{Sadiq2013}: $(x-x_k)^{n_k-r-s}$ should be $(x-x_k)^{r+s-n_k}$ and $(x-x_k)^{n_k-r}$ be $(x-x_k)^{r-n_k}$.}
\begin{equation}\label{bary form2}
  H_{N-1}(f,x)=\frac{\sum_{k=1}^n\sum_{s=0}^{n_k-1}\frac{f_{k,s}}{s!}\sum_{r=0}^{n_k-s-1}w_{k,r}(x-x_k)^{r+s-n_k}}{\sum_{k=1}^n\sum_{r=0}^{n_k-1}w_{k,r}(x-x_k)^{r-n_k}}
\end{equation}
(see e.g. \cite{Butcher2011,Sadiq2013}).

The second barycentric form is more robust in the presence of rounding
errors in the weights $w_{k,r}$. It is obvious from inspection that either of the two forms can be used to evaluate the
interpolant $H_{N-1}(f,x)$ at a given point $x$ using $O(\sum_{k=1}^n n_k^2)$ arithmetic operations once the barycentric weights $w_{k,r}$ are known.

Schneider and
Werner \cite{Schneider} used divided differences to evaluate the barycentric weights.  This method
requires $n(n - 1)/2$ divisions, $(N^2-\sum_{k=1}^n n_k^2)/2$ multiplications and  about the same number of subtractions or additions \cite{Sadiq2013}. However, the numerical stability depends upon a good ordering of the grid points as mentioned above. Moreover, Newton interpolation requires the recomputation of the divided difference tableau for
each new function.

Butcher et al. \cite{Butcher2011} introduced an efficient method, compared with that in \cite{Schneider}, for computing the barycentric
weights, which is derived by using contour integrals and
the manipulation of infinite series. More recently, Sadiq and Viswanath \cite{Sadiq2013} gave another more direct and simple derivation of this method: Calculating the barycentric weights is to find the coefficients
in the Taylor polynomial of expressions of the form
$$
\prod_{j\not=k}(x + x_k - x_j)^{-n_j}=\prod_{j\not=k}(x_k - x_j)^{-n_j}\prod_{j\not=k}\left(1-\frac{x-x_k}{x_k-x_j}\right)^{-n_j}
$$
which can be obtained by the following recursion
$$
I_{k,r} =\sum_{s=1}^rP_{k,s}I_{k,r-s}/r,\quad I_{k,0}=1,\quad P_{k,r} =\sum_{j\not=k}n_j(x_j-x_k)^{-r},\quad r=1,2,\ldots,n_k-1,
$$
and then $w_{k,r} = C_kI_{k,r}$ with $C_k =\prod_{j\not=k}(x_k - x_j)^{-n_j}$. It costs $O(n\sum_{k=1}^n n_k+\sum_{k=1}^n n_k^2)$ multiplications. Roughly
half these operations are additions or subtractions and roughly half are multiplications. Furthermore, if an additional derivative is prescribed at one of the interpolation points, update the barycentric coefficients use only $O(N)$ operations \cite{Sadiq2013}.

Notice that the barycentric Hermite interpolation problem is highly susceptible to overflows or
underflows. The weights $w_{k,r}$ in (1.10) and (1.11) usually vary by exponentially large factors.  Figures 1.1-1.2 illustrate the magnitudes of the barycentric weights computed by the method of Sadiq and Viswanath \cite{Sadiq2013} at the Chebyshev pointsystem (1.3) or Legendre pointsystem with different multiple number $m$ of derivatives  ($n_1=n_2=\cdots=n_n=m$). We can see from these two figures that the barycentric weights become extremely large while the number of points and the multiple number of derivatives are not so large, which will lead to overflows\footnote{In {\sc Matlab}, the largest positive normalized floating-point number in IEEE double precision is $ (1+(1-2^{-52}))2^{1023}\approx 1.7977\times10^{308}$, the smallest positive normalized floating-point number in IEEE double precision is $2^{-1022} \approx
2.225\times10^{-308}$.} for larger $n$ or $m$. Table 1.1 shows the threshold $S$ that the algorithm suffers overflows  for computation of the weights if $n\ge S$ with different $m$, respectively.

\begin{figure}[htbp]
\centerline{\includegraphics[width=1.8in]{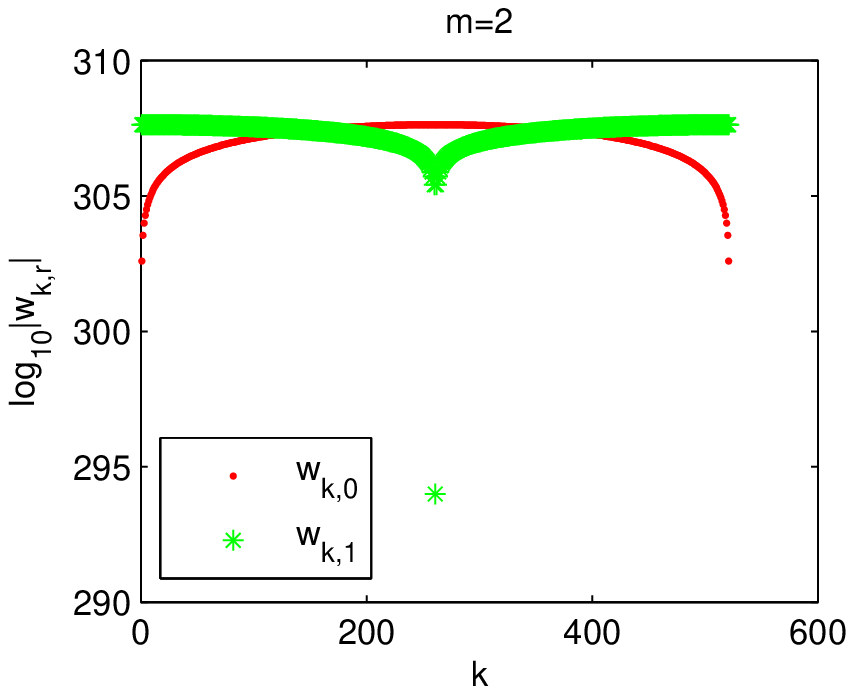} \includegraphics[width=1.8in]{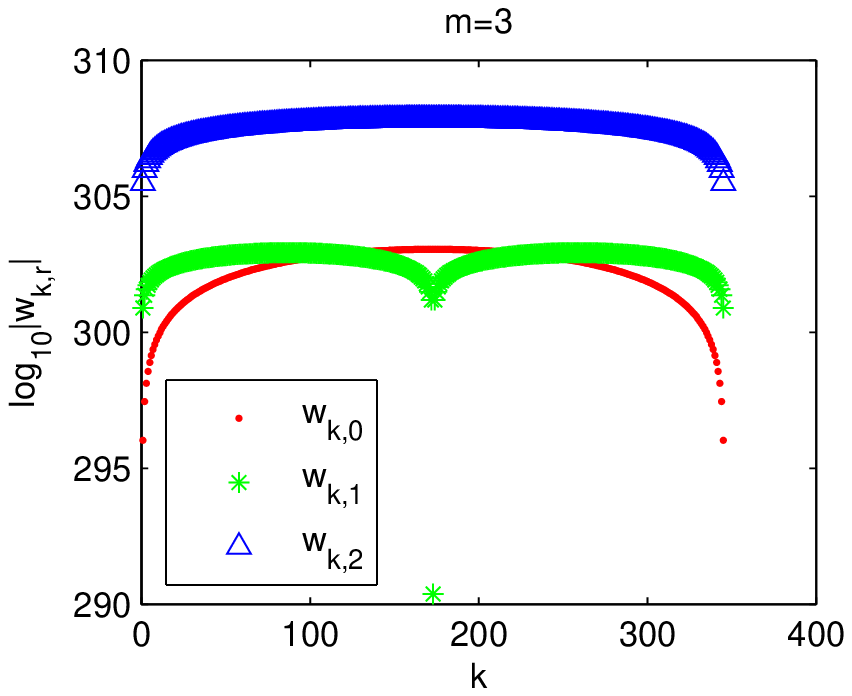}
  \includegraphics[width=1.8in]{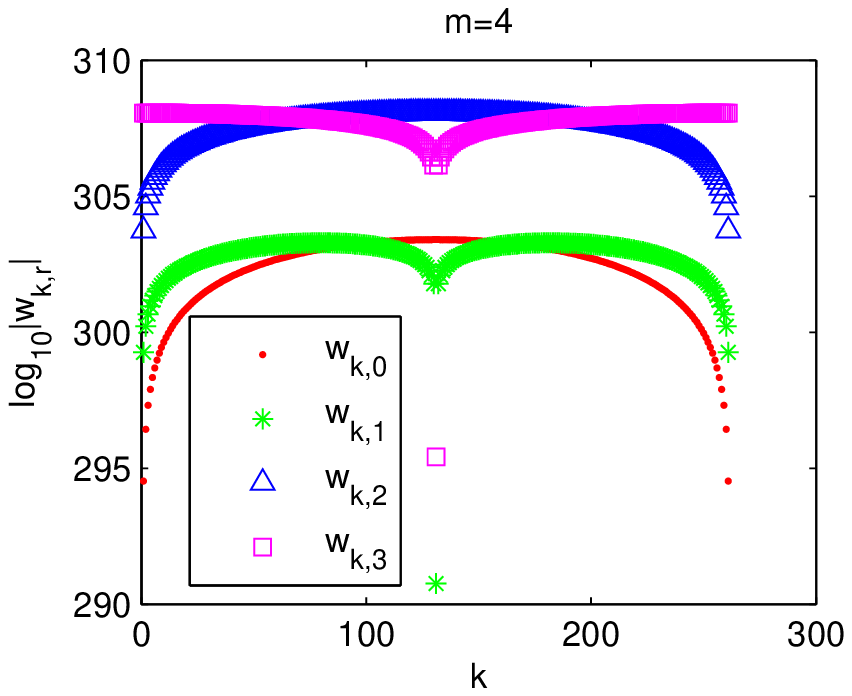}}
  \caption{Magnitudes of the barycentric weights $w_{k,r}$  by method of Sadiq and Viswanath interpolating at Chebyshev pointsystem (1.3).}\label{chebyshev points}
  \end{figure}

\begin{figure}[htbp]
\centerline{\includegraphics[width=1.8in]{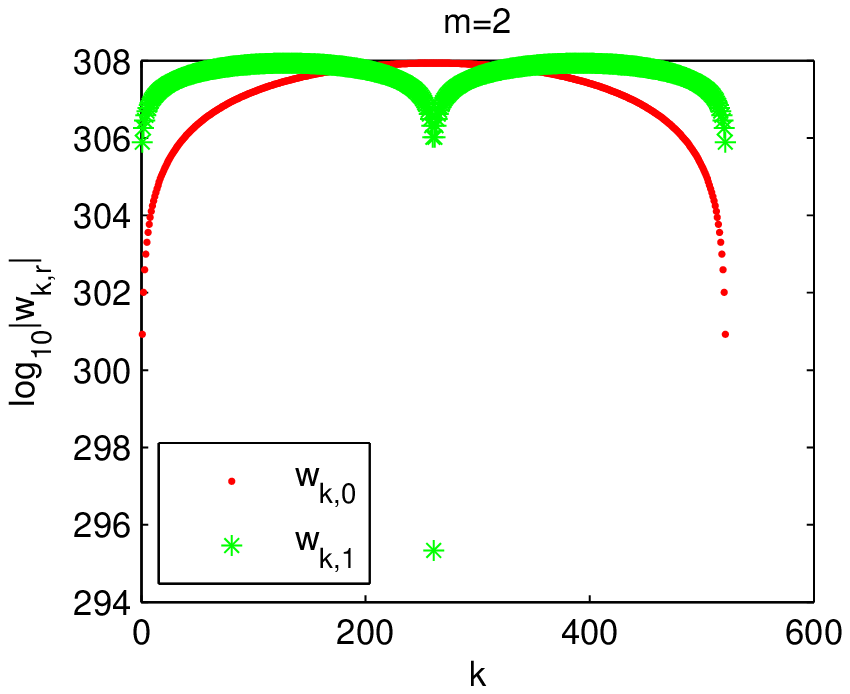} \includegraphics[width=1.8in]{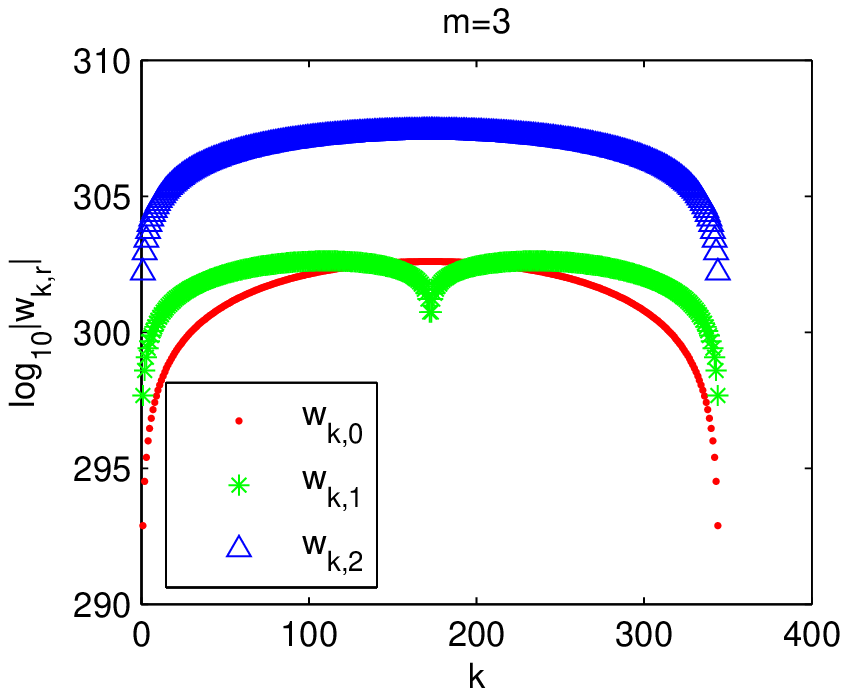}
  \includegraphics[width=1.8in]{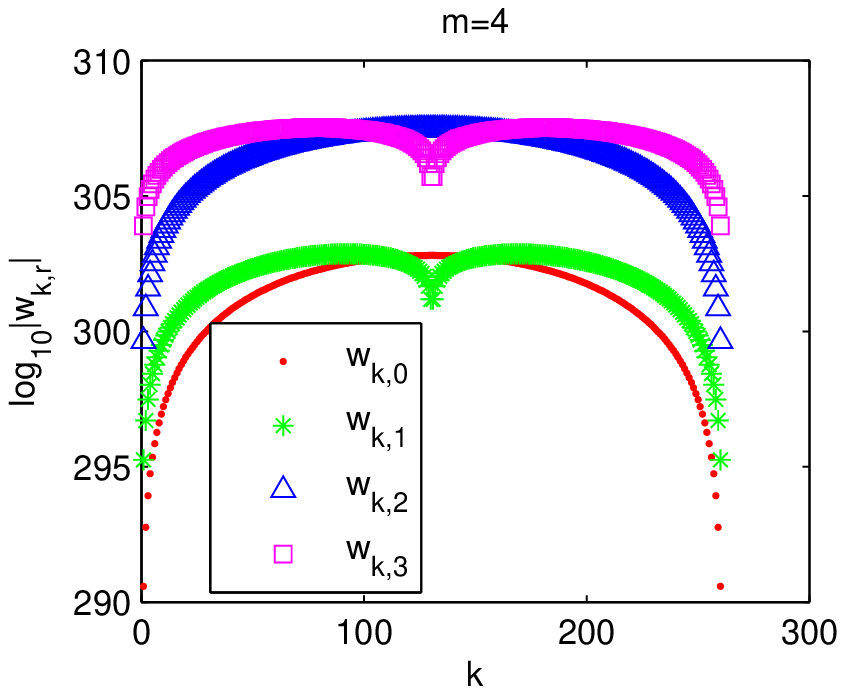}}
  \caption{Magnitudes of the barycentric weights $w_{k,r}$  by method of Sadiq and Viswanath interpolating at Legendre pointsystem.}
\end{figure}

\begin{table}[!ht]
\tabcolsep 0pt \caption{The threshold $S$ that the algorithm \cite{Sadiq2013} collapses for computation of the barycentric weights if $n\ge S$ with different $m$ at the Chebyshev pointsystem (1.3) and Legendre pointsystem} \vspace*{-6pt}
\begin{center}{\small
\def\temptablewidth{0.90\textwidth}
{\rule{\temptablewidth}{1pt}}
\begin{tabular*}{\temptablewidth}{@{\extracolsep{\fill}}ccc|ccc}
 &{\rm Chebyshev pointsystem (1.3)}&     &    &{\rm Legendre pointsystem}&\\  \hline

     $m=2$              & $m=3$            &$m=4$                & $m=2$              & $m=3$            &$m=4$               \\   \hline
  S=524& S=347 & S=263 &  S=523 &   S=346  &  S=262   \\

       \end{tabular*}
       {\rule{\temptablewidth}{1pt}}}
       \end{center}
       \end{table}

Thus, in \cite{Sadiq2013}, Sadiq and Viswanath used $2\cos((2k-1)\pi/(2n))$ instead of $\cos((2k-1)\pi/(2n))$ and considered Leja reordering of the points to get more stable computation of the weights and decrease the chance of overflows or underflows. However, reordering of the points  does not change the magnitudes of the barycentric weights. Furthermore, Leja reordering needs $O(n^2)$ operations \cite{Calvetti}.

Fortunately, from the second barycentric form (1.11), we see that the weights $w_{k,r}$  appear in the denominator exactly as in the
numerator. Due to the division,
the barycentric weights can be simplified by cancelling the common factors without altering
the result (see Figures 2.1-2.4 and Tables 2.1-2.4 below).

In this paper, we are concerned with fast implementation of the higher order Hermite-Fej\'{e}r interpolation polynomial (1.8), based on the second barycentric form (1.11), at Gauss-Jacobi or Jacobi-Gauss-Lobatto pointsystems.

Recently, a new algorithm on the evaluation of the nodes and weights for the Gauss quadrature was
given by Glaser, Liu and Rokhlin \cite{Glaser} with $O(n)$ operations, which has been extended by both Bogaert, Michiels and Fostier
\cite{Bogaert}, and Hale and Townsend \cite{HaleTownsend}. A {\sc Matlab} routine for computation of these nodes and weights can be found in {\sc Chebfun}
system \cite{Chebfun}.

As a result of these
developments,  in Section 2, we will discuss in details on calculation of the second barycentric weights of the higher order Hermite-Fej\'{e}r interpolation polynomial (1.11) at Gauss-Jacobi or Jacobi-Gauss-Lobatto pointsystems, and present two algorithms with $O(nm^2)$ operations. Particularly, due to division, the common factor in the barycentric weights can be canceled, which yields a superiorly stable method for computing the simplified barycentric
weights. In Section 3,
we will consider the stabilities of these implementations for the second barycentric formula on higher order Hermite-Fej\'{e}r interpolation and present numerical examples illustrating the efficiency and accuracy. A final remark on the convergence rate of Hermite-Fej\'{e}r interpolation (1.5) is included in Section 4.

All the numerical results in this paper are carried out by using {\sc Matlab} R2012a on a  desktop (2.8 GB RAM, 2 Core2 (32 bit)
processors at 2.80 GHz) with Windows XP operating system.


\vspace{0.1cm}
\section{Fast computation of the barycentric weights on higher order Hermite-Fej\'{e}r interpolation}

In this section, we will introduce two methods for fast computation of the barycentric weights on higher order Hermite-Fej\'{e}r interpolation at Gauss-Jacobi or Jacobi-Gauss-Lobatto pointsystems, which both share $O(nm^2)$ operations and lead to fast and stable calculation of the barycentric weights due to the division in (1.11), where the exponentially increasing common factor is cancelled.

$\mathbf{Algorithm~ 1}$: Following the barycentric Hermite interpolation formula in \cite{Butcher2011,Sadiq2013,Schneider}, we rewrite the higher order Hermite-Fej\'{e}r interpolation (1.8) as
\begin{equation}\label{bary form}\begin{array}{lll}
  H_{mn-1}(f,x) &=&{\displaystyle H^{*}_{mn}(x)\sum_{k=1}^n\frac{f_{k,m-1}}{(m-1)!}\left(\frac{w_{k,0}}{x-x_k}\right)+\frac{f_{k,m-2}}{(m-2)!}\left(\frac{w_{k,0}}{(x-x_k)^2}+\frac{w_{k,1}}{x-x_k}\right) }\\
  &&+\cdots+f_{k,0}\left(\frac{w_{k,0}}{(x-x_k)^m}+\cdots+\frac{w_{k,m-1}}{x-x_k}\right)
\end{array}\end{equation}
with
\begin{equation}\label{hermite condition}
  \frac{d^rH_{mn-1}(f,x)}{dx^r}\bigg|_{x=x_k}=f_{k,r} \mbox{\quad for $r=0,\ldots,m-1$ and $k=1,\cdots,n$,}
\end{equation}
where $H_{mn}^{*}(x)=\prod_{k=1}^n(x-x_k)^m$ and $f_{k,j}=f^{(j)}(x_k)$. Denote by $\hbar_k(x)=\frac{\omega_n(x)}{x-x_k}=\omega_n'(x_k)\ell_k(x)$,
then, expression (\ref{bary form}) can be represented as
\begin{equation}{\small \begin{array}{lll}
H_{mn-1}(f,x) &=&{\displaystyle  \sum_{k=1}^nf_{k,0}\left(w_{k,0}\hbar^m_k(x)+w_{k,1}\hbar_k^m(x)(x-x_k)+\cdots+w_{k,m-1}\hbar^m_k(x)(x-x_k)^{m-1}\right)}\\
&& +\cdots+\frac{f_{k,m-1}}{(m-1)!}w_{k,0}\hbar^m_k(x)(x-x_k)^{m-1}.
\end{array}}\end{equation}
Furthermore, from (1.9) and (2.1)-(2.3), we have
\begin{equation}{\small\begin{array}{l}
  w_{k,0}\hbar_k^m(x)\bigg|_{x=x_k}=1, \\
  w_{k,0}(\hbar_k^m(x))^{(j)}+w_{k,1}(\hbar_k^m(x)(x-x_k))^{(j)}+\cdots+w_{k,j}(\hbar_k^m(x)(x-x_k)^j)^{(j)}\bigg|_{x=x_k}=0
\end{array}}\end{equation}
for $j=1,\ldots,m-1$.
By using the Taylor expansion of $\hbar_k^m(x)$ at $x_k$, it leads to $$\frac{\left(\hbar_k^m(x)(x-x_k)^i\right)^{(j)}}{j!}\bigg|_{x=x_k}=\frac{(\hbar_k^m(x))^{(j-i)}}{(j-i)!}\bigg|_{x=x_k},~~~i=0,1,\ldots,j-1.$$
Thus, from (2.4), we get the following formulas
\begin{equation}
\quad\,\,\,\, {\small \mbox{${\displaystyle w_{k,0}=\frac{1}{(\omega'_n(x_k))^m},\quad  w_{k,j}=-\sum_{i=0}^{j-1}\frac{w_{k,i}}{(j-i)!}\bigg(\ell_k^m(x)\bigg)^{(j-i)}\Bigg|_{x=x_k},~j=1,2,\ldots,m-1.}$}}
\end{equation}

In addition, from the definition of $\ell_k(x)=\frac{\omega_n(x)}{\omega'_n(x_k)(x-x_k)}$, it is not difficult to deduce by applying the Taylor expansion of $\omega_n(x)$ at $x=x_k$ that
\begin{equation} \frac{\left(\ell_k(x)\right)^{(r)}\big|_{x=x_k}}{r!}=\frac{\omega_n^{(r+1)}(x_k)}{(r+1)!\omega'_n(x_k)}:=M_{k,r}, \quad r=0,\cdots,m-1. \end{equation}
Set $b_{k,j}=\frac{\big[\ell_k^m(x)\big]_{x=x_k}^{(j)}}{j!}$ $(j=0,1,\ldots,m-1)$ and $a_{k,i}=\frac{m}{(i-1)!}\big[\frac{\ell'_k(x)}{\ell_k(x)}\big]_{x=x_k}^{(i-1)}$ ($i=1,2,\ldots,m-1$). Noting that $b_{k,j}=\frac{m}{j!}\big[\ell_k^{m-1}(x)\ell'_k(x)\big]_{x=x_k}^{(j-1)}=\frac{m}{j!}\big[\ell_k^{m}(x)\frac{\ell'_k(x)}{\ell_k(x)}\big]_{x=x_k}^{(j-1)}$, it follows by Leibniz formula that
\begin{equation}
  b_{k,j} = \frac{1}{j}\sum_{i=1}^ja_{k,i}b_{k,j-i},\quad j=1,2,\ldots,m-1; \quad b_{k,0}=1.
\end{equation}

In the following, we shall show that $a_{k,i}$ can be computed from $M_{k,r}$, then $b_{k,j}$ can be evaluated from the recursion (2.7). In fact, $a_{k,i}$ can be calculated from the coefficient of the Taylor expansion of $\ell'_k(x)/\ell_k(x)$ at $x=x_k$ by the following lemma.

\begin{lemma}
  Let $F(x)=\frac{A(x)}{B(x)}=\sum_{j=1}^{s}F_j(x-x_k)^{j-1}+O((x-x_k)^{s})$, where $A(x)=\sum_{j=1}^{s}A_j(x-x_k)^{j-1}+O((x-x_k)^{s})$ and $B(x)=\sum_{j=1}^{s}B_j(x-x_k)^{j-1}+O((x-x_k)^{s}) $, then it obtains
  \begin{equation}
   B_1F_j=A_j-\sum_{i=2}^{j}B_iF_{j-i+1},~~~j=1,\cdots,s.
  \end{equation}
\end{lemma}
\begin{proof}
  The proof is trivial.
\end{proof}

Since
\begin{equation}
 \ell_k(x) = \sum_{r=1}^{m-1}M_{k,r-1}(x-x_k)^{r-1}+O((x-x_k)^{m-1}),
\end{equation}
and
\begin{equation}
 \ell'_k(x) = \sum_{r=1}^{m-1}rM_{k,r}(x-x_k)^{r-1}+O((x-x_k)^{m-1}),
\end{equation}
by Lemma 2.1 and using $M_{k,0}=1$ and $\frac{\ell_k'(x)}{\ell_k(x)}=\frac{1}{m}\sum_{r=1}^{m-1}a_{k,r}(x-x_k)^{r-1}+O((x-x_k)^{m-1})$, we have
  \begin{equation}
a_{k,1}=mM_{k,1}, ~~~a_{k,i}=imM_{k,i}-\sum_{j=2}^{i}M_{k,j-1}a_{k,i-j+1},~~~i=2,\ldots,m-1,
  \end{equation}
and then by (2.5) we get
\begin{equation}
\quad\,\,\,\, w_{k,0}=\frac{1}{(\omega'(x_k))^m},\quad  w_{k,j}=-\sum_{i=0}^{j-1}w_{k,i}b_{k,j-i},~j=1,\ldots,m-1.
\end{equation}
Thus, if $\{M_{k,r}\}_{r=0}^{m-1}$ and $w_{k,0}$ are known, from (2.11) the total computation of $\{a_{k,j}\}_{j=0}^{m-1}$ costs $O(m^2)$ operations, the same as those for $\{b_{k,j}\}_{j=0}^{m-1}$ and $\{w_{k,j}\}_{j=0}^{m-1}$ by (2.7) and (2.12), respectively, and then the implementation of (1.11) with $n_1=\cdots=n_n=m$ costs $O(nm^2)$ operations.

\vspace{0.36cm}
$\mathbf{Algorithm~ 2}$: The barycentric weights $w_{k,r}$ of the Hermite-Fej\'{e}r interpolation (2.1)  can also be calculated by another fast way based on the formula given in Szabados \cite{Sza1993},
\begin{equation}\label{HermiteFejerHigh}
H_{mn-1}(f,x)=\sum_{k=1}^n\sum_{j=0}^{m-1}f^{(j)}(x_k)\frac{\ell_k(x)^m}{j!}\sum_{i=0}^{m-j-1}\frac{[\ell_k(x)^{-m}]^{(i)}_{x=x_k}}{i!}(x-x_k)^{i+j}.
\end{equation}

We rewrite the interpolation (2.13) in the first barycentric interpolation form
 \begin{equation}\quad\quad
{\small \mbox{${\displaystyle H_{mn-1}(f,x)=H^*_{mn}(x)\sum_{k=1}^n\frac{1}{(\omega'_n(x_k))^m}\sum_{j=0}^{m-1}\frac{f_{k,j}}{j!}\sum_{i=0}^{m-j-1}
\frac{\big[\ell_k(x)^{-m}\big]^{(i)}_{x=x_k}}{i!}(x-x_k)^{i+j-m},}$}}
\end{equation}
and  second barycentric interpolation form
 \begin{equation}\quad\quad
{\small \mbox{${\displaystyle H_{mn-1}(f,x)=\frac{\sum_{k=1}^nw_{k,0}\sum_{j=0}^{m-1}\frac{f_{k,j}}{j!}\sum_{i=0}^{m-j-1}
\frac{\big[\ell_k(x)^{-m}\big]^{(i)}_{x=x_k}}{i!}(x-x_k)^{i+j-m}}{\sum_{k=1}^n
w_{k,0}\sum_{i=0}^{m-1}\frac{\big[\ell_k(x)^{-m}\big]^{(i)}_{x=x_k}}{i!}(x-x_k)^{i-m}}}$}}
\end{equation}
respectively. Next we concentrate on the computation of $\tilde{b}_{k,j}: = \frac{\big[\ell_k(x)^{-m}\big]^{(j)}_{x=x_k}}{j!}$ $(j=0,1,\ldots,m-1)$. Comparing (2.14) with (2.1), we find that the barycentric weights satisfies
\begin{equation} w_{k,j}=\tilde{b}_{k,j}w_{k,0},\quad\quad j=0,1,\ldots,m-1. \end{equation}

From \cite{Sza1993}, it follows
\begin{equation}
 \tilde{b}_{k,j}=\frac{1}{j}\sum_{i=1}^j\tilde{a}_{k,i}\tilde{b}_{k,j-i},\quad j=1,2,\ldots,m-1;\quad \tilde{b}_{k,0}=1,
\end{equation}
where $\tilde{a}_{k,i}=\frac{m}{(i-1)!}\big[\frac{1}{x-x_k}-\frac{\omega'_n(x)}{\omega_n(x)}\big]^{(i-1)}_{x=x_k}$ ($i=1,\cdots,m-1$). Then from $\omega_n(x)=\omega'_n(x_k)\ell_k(x)(x-x_k)$, it is easy to see that $\tilde{a}_{k,i}=\frac{m}{(i-1)!}\big[\frac{\ell_k(x)\omega'_n(x_k)-\omega'_n(x)}{\omega_n(x)}\big]^{(i-1)}_{x=x_k}=\frac{m}{(i-1)!}\big[\frac{-\ell'_k(x)}{\ell_k(x)}\big]^{(i-1)}_{x=x_k}$.
Similarly, by Lemma 2.1 and (2.9)-(2.10), we get
\begin{equation}
   \tilde{a}_{k,j}=-mM_{k,1},~~~\tilde{a}_{k,j}=-jmM_{k,j}-\sum_{i=2}^{j}M_{k,j-1}\tilde{a}_{k,j-i+1},~ j=2,\ldots,m-1.
\end{equation}
Thus, from (2.18) and (2.17), the total computation of barycentric weights $\{\widetilde{b}_{k,r}\}_{r=0}^{m-1}$ costs also $O(m^2)$ operations if $\{M_{k,r}\}_{r=0}^{m-1}$ and $w_{k,0}$ are known.

From the above illustrations, we see that fast computation of $M_{k,r}$ and $w_{k,0}$ leads to fast implementation of higher order barycentric Hermite-Fejer interpolation. We shall show that for each $k$, $\{M_{k,r}\}_{r=0}^{m-1}$ can be rapidly calculated with $O(m)$ operations for Gauss-Jacobi or Jacobi-Gauss-Lobatto pointsystems.


\subsection{Gauss-Jacobi pointsystems}
Let $\{x_k\}_{k=1}^n$ be the zeros of the Jacobi polynomial $P_n^{(\alpha,\beta)}(x)$. Thus $\omega_n(x)=\frac{P_n^{(\alpha,\beta)}(x)}{K_n}$, where $K_n$ is the leading coefficient of $P_n^{(\alpha,\beta)}(x)$. From (2.6), we get $$M_{k,r-1}=\frac{\omega_n^{(r)}(x_k)}{r!\omega_n'(x_k)}=\frac{\frac{d^rP_n^{(\alpha,\beta)}}{dx^r}(x_k)}{r!P{'}_n^{(\alpha,\beta)}(x_k)}, ~~ r=1,\ldots,m.$$
It is known that $P_n^{(\alpha,\beta)}(x)$ is the unique solution of the second order linear homogeneous Sturm-Liouville differential equation
\begin{equation}
  (1-x^2)y''+(\beta-\alpha-(\alpha+\beta+2)x)y'+n(n+\alpha+\beta+1)y=0,
\end{equation}
from which it is not difficult to deduce that
\begin{equation}\label{jacobi_diff}\quad\begin{array}{l}
(1-x^2)y^{(r+2)}+[\beta-\alpha-(\alpha+\beta+2(r+1))x]y^{(r+1)}+[n(n+\alpha+\beta+1)\\
\hspace{6cm} -r(\alpha+\beta+r+1)]y^{(r)}=0,\end{array}
\end{equation}
for $r=0,1,\ldots$. Thus, we have
\begin{equation}\quad
  \mbox{$M_{k,r+1}=\frac{(\alpha+\beta+2(r+1))x_k+\alpha-\beta}{1-x_k^2}\frac{1}{(r+2)}M_{k,r}+\frac{r(\alpha+\beta+r+1)-n(n+\alpha+\beta+1)}{1-x_k^2}\frac{1}{(r+2)(r+1)}M_{k,r-1}$},
\end{equation}
with $M_{k,0}=1$, $M_{k,1}=\frac{\alpha-\beta+(\alpha+\beta+2)x_k}{2(1-x_k^2)}$ and then $\{M_{k,r}\}_{r=0}^{m-1}$ can be computed in $O(m)$. Consequently, from (2.6)-(2.7), (2.11)-(2.12), (2.14)-(2.15) and (2.18), the barycentric form (1.11) with $n_1=\cdots=n_n=m$ and (2.15)  can be achieved in $O(nm^2)$ operations if $w_{k,0}$ is known.

From Wang et al. \cite{Wang2013}, $w_{k,0}$ has the explicit form
\begin{equation}
w_{k,0}=\left[C_n^{(\alpha,\beta)}(-1)^{k+1}\sqrt{(1-x_k^2)\overline{w}_k}\right]^m,\quad k=1,2,\ldots,n,
\end{equation}
where $\overline{w}_k$ is the Gaussian quadrature weight corresponding to $x_k$ for the Jacobi weight,
$$
C_n^{(\alpha,\beta)}=\sigma_n\frac{\Gamma(2n+\alpha+\beta+1)}{2^{n+\frac{\alpha+\beta+1}{2}}}\frac{1}{\sqrt{n!\Gamma(n+\alpha+\beta+1)\Gamma(n+\alpha+1)\Gamma(n+\beta+1)}}
$$
and $\sigma_n=+1$ for $n$ odd and  $\sigma=-1$ for $n$ even. Moreover, both $\{x_k\}_{k=1}^n$ and $\{\overline{w}_k\}_{k=1}^n$ can be efficiently calculated by routine \textbf{jacpts} in {\sc Chebfun} \cite{Chebfun} with $O(n)$ operations.

Additionally, it is worth noting that due to the division of the second barycentric form (1.11) or formula (2.15), the common factor $(C_n^{(\alpha,\beta)})^m$ of weights $w_{k,r}$ from (2.12) and (2.16) can be cancelled without affecting the value of $H_{mn-1}(f,x)$ (we call these new weights as simplified barycentric weights). Then the barycentric weight $w_{k,0}$ can be simplified as
\begin{equation}
w_{k,0}=(-1)^{m(k+1)}\left[\sqrt{(1-x_k^2)\overline{w}_k}\right]^m.
\end{equation}

Comparing these two algorithms with Sadiq and Viswanath's \cite{Sadiq2013}, we find that the new algorithms cost $O(nm^2)$ operations much less than that $O(n^2m+nm^2)$ given by Sadiq and Viswanath \cite{Sadiq2013} if $n\gg m$. Moreover, due to the cancellation of the common factor $(C_n^{(\alpha,\beta)})^m$, the computation of $\{w_{k,j}\}_{j=0}^{m-1}$ is quite efficient and stable (see Figures 2.1-2.2 and Tables 2.3-2.4). Tables 2.1-2.2 show the values of the common factor $\left(C_n^{(\alpha,\beta)}\right)^m$ with respect to Chebyshev pointsystem (1.3) and Gauss-Legendre pointsystem, respectively.

\begin{table}[!ht]
\tabcolsep 0pt \caption{$\left(C_n^{(\alpha,\beta)}\right)^m$ with respect to Chebyshev pointsystem (1.3): $\alpha=\beta=-0.5$} \vspace*{-6pt}
\begin{center}{\small
\def\temptablewidth{.950\textwidth}
{\rule{\temptablewidth}{1pt}}
\begin{tabular*}{\temptablewidth}{@{\extracolsep{\fill}}ccccc}

 &    $m=2$              & $m=3$            &$m=4$        &$m=10$          \\    \hline
$n=100$ &$1.27876*10^{57}$  & $-4.57282*10^{85}$ & $1.63523*10^{114}$  &$3.41937*10^{285}$        \\

$n=200$ &$1.02744*10^{117}$ & $-3.29335*10^{175}$& $1.05564*10^{234}$  &$1.14496*10^{585}$   \\

$n=500$ &$1.70536*10^{297}$ & $-7.04245*10^{445}$& $2.90825*10^{594}$  &$1.44238*10^{1486}$\\

$n=1000$ & $9.13653*10^{597}$ & $-8.73318*10^{896}$ &  $8.34762*10^{1195}$  &$6.36660*10^{2989}$\\

       \end{tabular*}
       {\rule{\temptablewidth}{1pt}}}
       \end{center}
       \end{table}

       \begin{table}[!ht]
\tabcolsep 0pt \caption{$\left(C_n^{(\alpha,\beta)}\right)^m$ with respect to  Gauss-Legendre pointsystem: $\alpha=\beta=0$} \vspace*{-6pt}
\begin{center}{\small
\def\temptablewidth{.950\textwidth}
{\rule{\temptablewidth}{1pt}}
\begin{tabular*}{\temptablewidth}{@{\extracolsep{\fill}}ccccc}

 &    $m=2$              & $m=3$            &$m=4$        &$m=10$          \\    \hline
$n=100$ &$2.55114*10^{57}$  & $-1.28855*10^{86}$ & $6.50829*10^{114}$  &$1.08061*10^{287}$        \\

$n=200$ &$2.05232*10^{117}$ & $-9.29755*10^{175}$& $4.21203*10^{234}$  &$3.64106*10^{586}$\\

$n=500$ &$3.40901*10^{297}$ & $-1.99041*10^{446}$& $1.16214*10^{595}$  &$4.60408*10^{1487}$\\

$n=1000$ &$1.82685*10^{598}$ & $-2.46919*10^{897}$& $3.33738*10^{1196}$  &$2.03477*10^{2991}$
       \end{tabular*}
       {\rule{\temptablewidth}{1pt}}}
       \end{center}
       \end{table}

\begin{figure}[htbp]
\centerline{\includegraphics[width=1.8in]{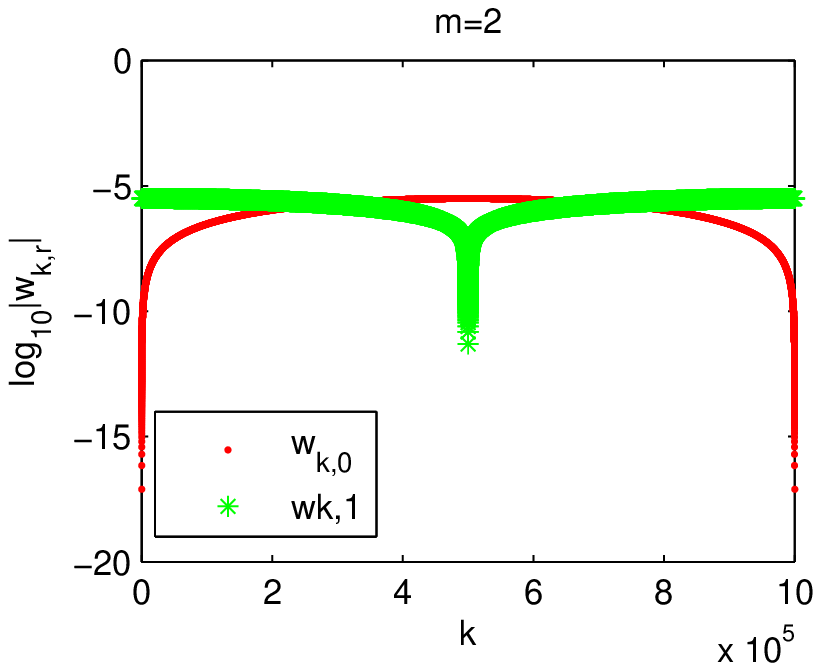} \includegraphics[width=1.8in]{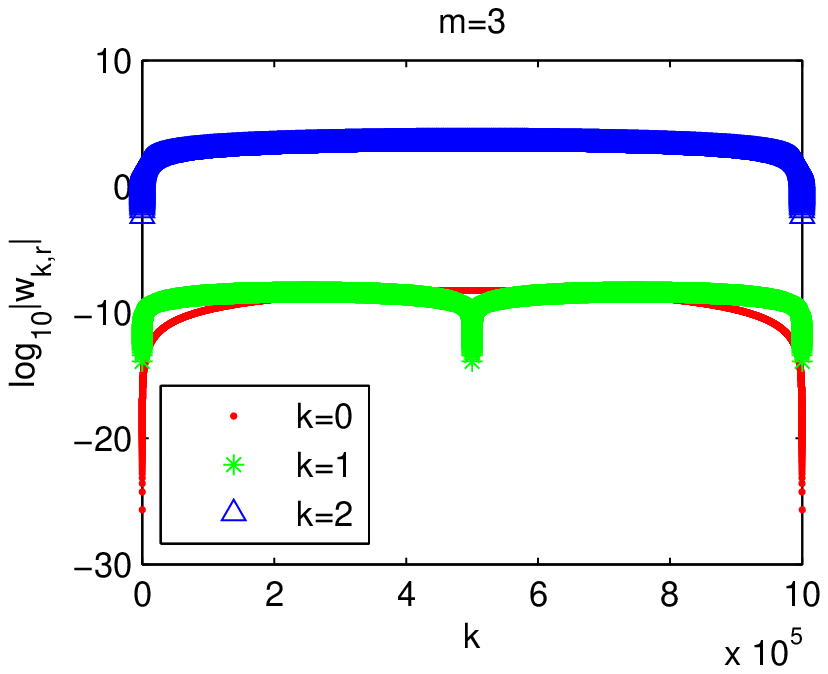}
  \includegraphics[width=1.8in]{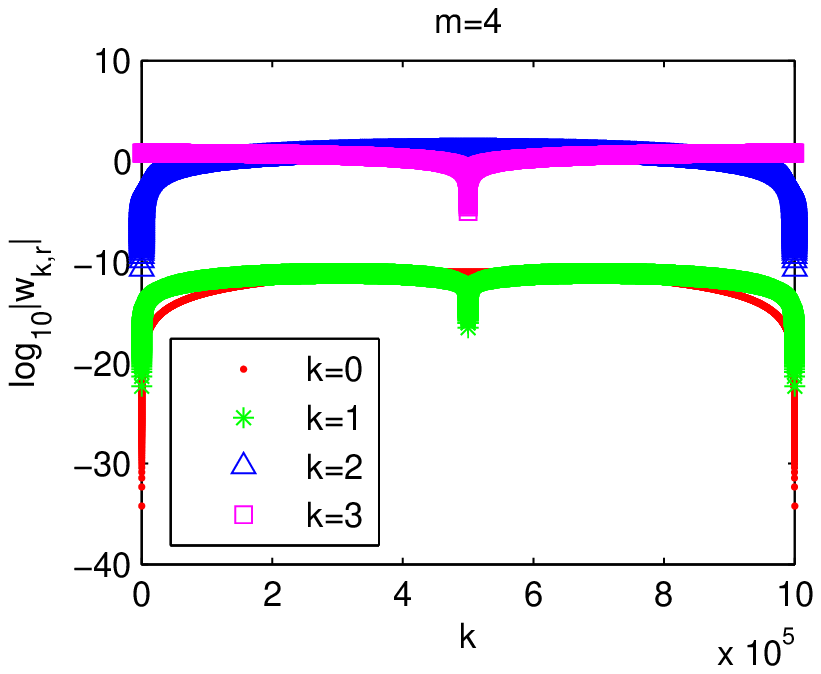}}
  \caption{Magnitude of the simplified barycentric weights $w_{k,r}$  by the $\mathbf{Algorithm~ 2}$ interpolating at the Chebyshev pointsystem (1.3): $k=1:10^6$ and $r=0:m-1$.}\label{modified chebyshev points}
\end{figure}

\begin{figure}[htbp]
\centerline{\includegraphics[width=1.8in]{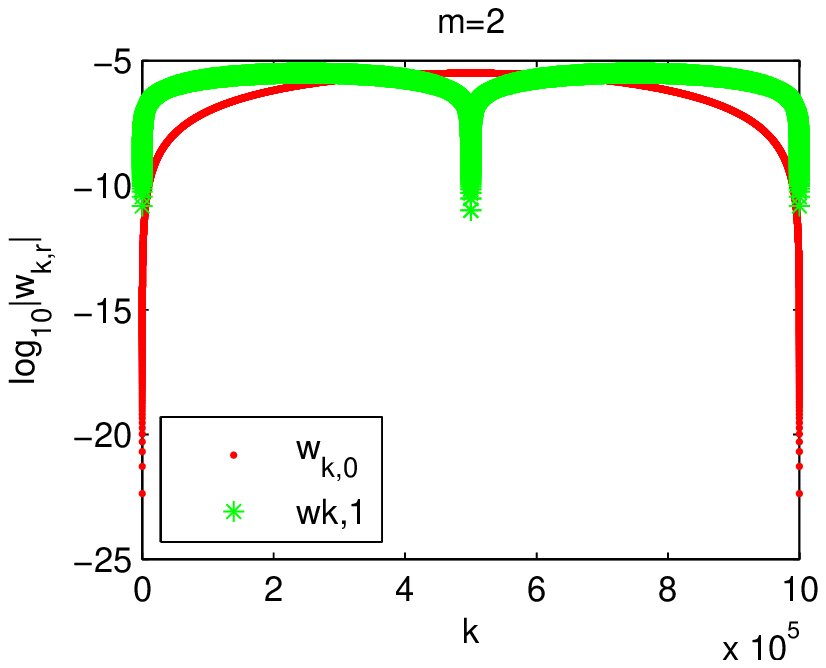} \includegraphics[width=1.8in]{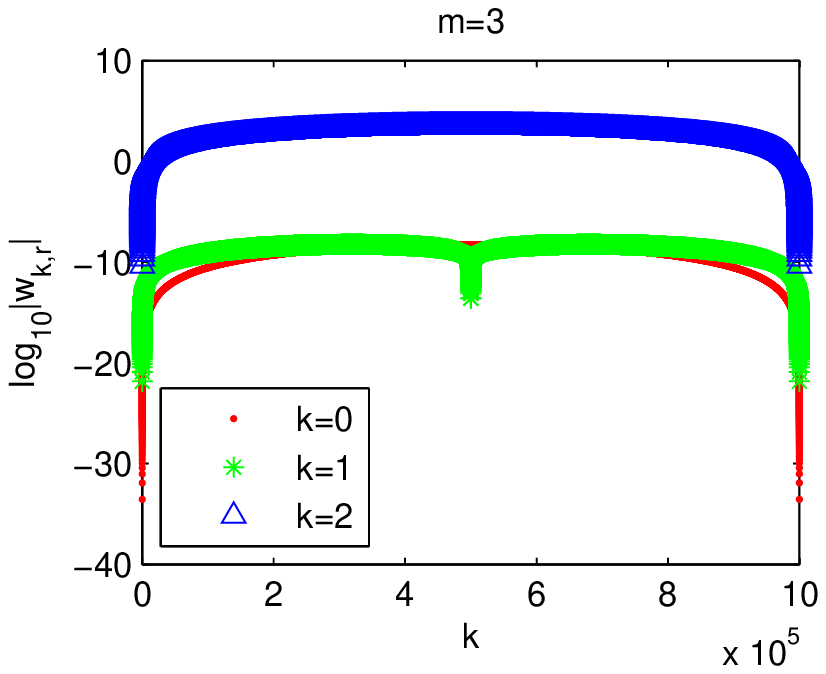}
  \includegraphics[width=1.8in]{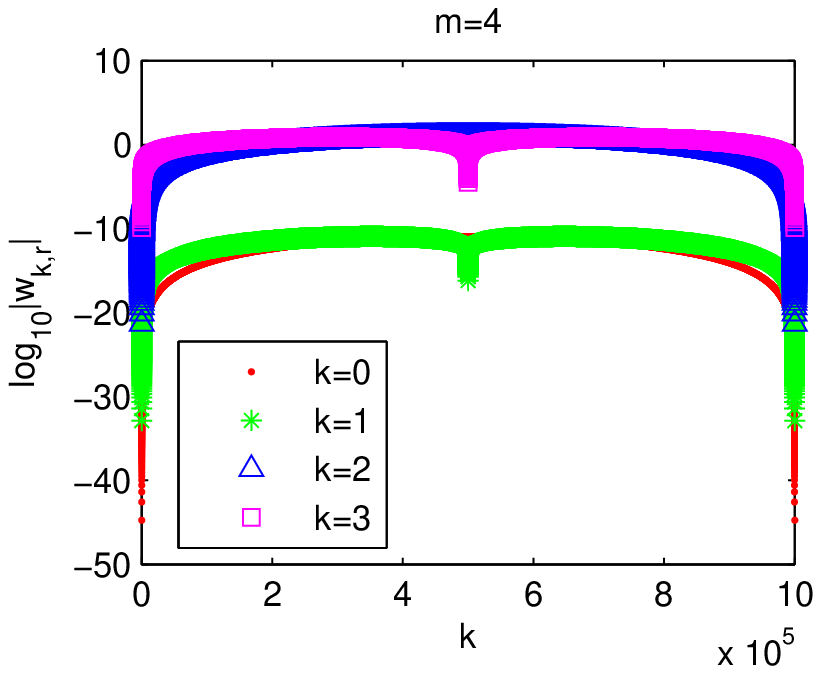}}
  \caption{Magnitude of the simplified barycentric weights $w_{k,r}$  by the $\mathbf{Algorithm~ 2}$ interpolating at Legendre pointsystem: $k=1:10^6$ and $r=0:m-1$.}\label{modified legendre points}
\end{figure}

\begin{table}[!ht]
\tabcolsep 0pt \caption{The CPU time for computation of the simplified barycentric weights $w_{k,r}$  by the $\mathbf{Algorithm~ 2}$ at the Chebyshev pointsystem (1.3): $k=1:n$ and $r=0:m-1$.} \vspace*{-6pt}
\begin{center}{\small
\def\temptablewidth{0.75\textwidth}
{\rule{\temptablewidth}{1pt}}
\begin{tabular*}{\temptablewidth}{@{\extracolsep{\fill}}ccccc}
     &$m=2$              & $m=3$            &$m=4$                &$m=10$\\   \hline
$n=10^4$  &0.011127s   &0.012392s  &  0.012778s  &  0.026214s\\

$n=10^5$  &  0.031632s   &  0.046582s   &   0.061328s   & 0.271533s    \\

$n=10^6$ & 0.239484s   & 0.377926s    &  0.530127s      & 2.813712s  \\
       \end{tabular*}
       {\rule{\temptablewidth}{1pt}}}
       \end{center}
       \end{table}

\begin{table}[!ht]
\tabcolsep 0pt \caption{The CPU time for computation of the simplified  barycentric weights $w_{k,r}$  by the $\mathbf{Algorithm~ 2}$ at the Legendre pointsystem: $k=1:n$ and $r=0:m-1$.} \vspace*{-6pt}
\begin{center}{\small
\def\temptablewidth{0.75\textwidth}
{\rule{\temptablewidth}{1pt}}
\begin{tabular*}{\temptablewidth}{@{\extracolsep{\fill}}ccccc}
   &  $m=2$              & $m=3$            &$m=4$                &$m=10$\\   \hline

$n=10^4$  &  0.166094   &   0.169577s     & 0.170217s& 0.182489s\\

$n=10^5$& 1.713747s   &1.712082s   &1.734930s   &   1.970211s

\\

$n=10^6$&17.445534s   &  17.745105s    &  18.010896s      & 20.098777s  \\
      \end{tabular*}
       {\rule{\temptablewidth}{1pt}}}
       \end{center}
       \end{table}

\newtheorem{remark}[theorem]{Remark}
%
%
%
%

\begin{remark}
The simplified barycentric weights (2.23), in the case $m=1$ ($w_{k,0}=(-1)^{k+1}\sqrt{(1-x_k^2)\overline{w}_k}$), is exact the barycentric weight for the barycentric formula of Lagrange interpolation at the Jacobi pointsystem $\{x_k\}_{k=1}^n$,  which was derived for Chebyshev points of second kind in Henrici \cite{Henricibook1982}, for Legendre points in Wang and Xiang \cite{WangXiang2012}, and extended to Jacobi points in Hale and Trefethen \cite{Hale2012}.
\end{remark}

\subsection{Jacobi-Gauss-Lobatto pointsystems}

Suppose $-1=x_1<x_2<\cdots<x_{n-1}<x_{n}=1$ are the $n$ zeros of $(x^2-1)P_{n-2}^{(\alpha,\beta)}(x)$. Thus from (2.6) we get
\begin{equation}\quad\,\, {\small
  \mbox{$M_{r-1}(x)=\frac{\omega_n^{(r)}(x)}{r!\omega_n'(x)}=\frac{(x^2-1)\left(P_{n-2}^{(\alpha,\beta)}\right)^{(r)}(x)+2rx\left(P_{n-2}^{(\alpha,\beta)}\right)^{(r-1)}(x)
  +r(r-1)\left(P_{{n-2}}^{(\alpha,\beta)}\right)^{(r-2)}(x)}{r!\left((x^2-1)\left(P_{{n-2}}^{(\alpha,\beta)}\right)'(x)+2xP_{{n-2}}^{(\alpha,\beta)}(x)\right)},
$}}\end{equation}
for $r=2\ldots,m$. Especially, $M_{k,0}=M_0(x_k)=1$ for $k=1,2,\ldots,n$.

For the case $x_k=\pm 1$, it follows
\begin{equation}
\quad\,\, {\small
  \mbox{$  M_{1,r-1}=M_{r-1}(-1)=\frac{\left(P_{n-2}^{(\alpha,\beta)}\right)^{(r-1)}(-1)}{(r-1)!P_{n-2}^{(\alpha,\beta)}(-1)}-\frac{1}{2}\frac{\left(P_{n-2}^{(\alpha,\beta)}\right)^{(r-2)}(-1)}
  {(r-2)!P_{n-2}^{(\alpha,\beta)}(-1)}, ~~ r=2,\ldots,m,$}}
\end{equation}
\begin{equation}
\quad\,\, {\small
  \mbox{$  M_{n,r-1}=M_{r-1}(1)=\frac{\left(P_{n-2}^{(\alpha,\beta)}\right)^{(r-1)}(1)}{(r-1)!P_{n-2}^{(\alpha,\beta)}(1)}+\frac{1}{2}\frac{\left(P_{n-2}^{(\alpha,\beta)}\right)^{(r-2)}(1)}
  {(r-2)!P_{n-2}^{(\alpha,\beta)}(1)}, ~~ r=2,\ldots,m,$}}
\end{equation}
and for the case $x_k\not=\pm1$ and $r=2,\ldots,m$,
\begin{equation}\quad\quad{\small\mbox{$
  M_{k,r-1}=M_{r-1}(x_k)=\frac{\left(P_{n-2}^{(\alpha,\beta)}\right)^{(r)}(x_k)}{r!\left(P_{n-2}^{(\alpha,\beta)}\right)'(x_k)}
  +\frac{2x_k}{x_k^2-1}\frac{\left(P_{n-2}^{(\alpha,\beta)}\right)^{(r-1)}(x_k)}{(r-1)!\left(P_{n-2}^{(\alpha,\beta)}\right)'(x_k)}
  +\frac{1}{x_k^2-1}\frac{\left(P_{n-2}^{(\alpha,\beta)}\right)^{(r-2)}(x_k)}{(r-2)!\left(P_{n-2}^{(\alpha,\beta)}\right)'(x_k)}$}},
\end{equation}
where
$\frac{\left(P_{n-2}^{(\alpha,\beta)}\right)^{(r)}(x_k)}{r!\left(P_{n-2}^{(\alpha,\beta)}\right)'(x_k)}$ can be evaluated by (2.21) with $n-2$ instead of $n$ for $x_k\not=\pm1$, and for $x_k=\pm 1$ by
$$\frac{\left(P_{n-2}^{(\alpha,\beta)}\right)^{(r)}(\pm1)}{r!\left(P_{n-2}^{(\alpha,\beta)}\right)(\pm1)}
=\frac{(n-2)(\alpha+\beta+n-1)-(r-1)(\alpha+\beta+r)}{r\left(\pm(\alpha+\beta+2r)+\alpha-\beta\right)}
\frac{\left(P_{n-2}^{(\alpha,\beta)}\right)^{(r-1)}(\pm1)}{(r-1)!\left(P_{n-2}^{(\alpha,\beta)}\right)(\pm1)}$$
for $r=1,2,\ldots,m-1$.

Moreover, from \cite{Wang2013}, $w_{k,0}$ has the explicit form of
\begin{equation}
\quad \quad w_{k,0}=\left[C_{n-2}^{(\alpha,\beta)}(-1)^{k+1}\sqrt{\eta_k\widehat{w}_k}\right]^m,\eta_k=\left\{\begin{array}{ll}
\beta,&k=1,\\
\alpha,&k=n,\\
1,&\mbox{otherwise},\end{array}\right.
\end{equation}
where $\widehat{w}_k=\frac{\overline{w}_k}{1-x_k^2}$ for $k=2,3,\ldots,n-1$,
$$
\widehat{w}_1=2^{\alpha+\beta-1}\frac{\Gamma(\beta)\Gamma(\beta+1)\Gamma(n+\alpha+1)n!}{\Gamma(n+\beta+1)\Gamma(n+\alpha+\beta+1)},\, \widehat{w}_{n}=2^{\alpha+\beta-1}\frac{\Gamma(\alpha)\Gamma(\alpha+1)\Gamma(n+\beta+1)n!}{\Gamma(n+\alpha+1)\Gamma(n+\alpha+\beta+1)}$$
and  $\{\overline{w}_k\}_{k=2}^{n-1}$ is the Gaussian quadrature weight corresponding to $\{x_k\}_{k=2}^{n-1}$. Due to the division in (2.15), the barycentric weight $w_{k,0}$ can be simplified as
\begin{equation}
w_{k,0}=(-1)^{m(k+1)}\left(\eta_k\widehat{w}_k\right)^{m/2}.
\end{equation}

Figures 2.3-2.4 show magnitude of the simplified barycentric weights $w_{k,r}$ by the $\mathbf{Algorithm~ 2}$ interpolating at the Jacobi-Gauss-Lobatto pointsystem  with $\alpha=\beta=1.5$ and Legendre-Gauss-Lobatto pointsystem with $n=10^6$ and $m=2,3,4$ respectively.

\begin{figure}[htbp]
\centerline{\includegraphics[width=1.8in]{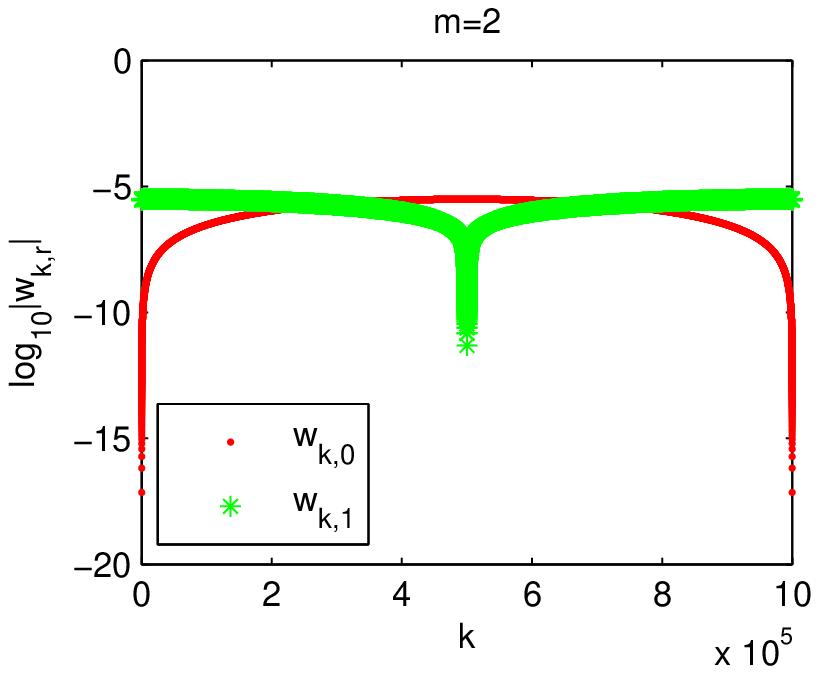} \includegraphics[width=1.8in]{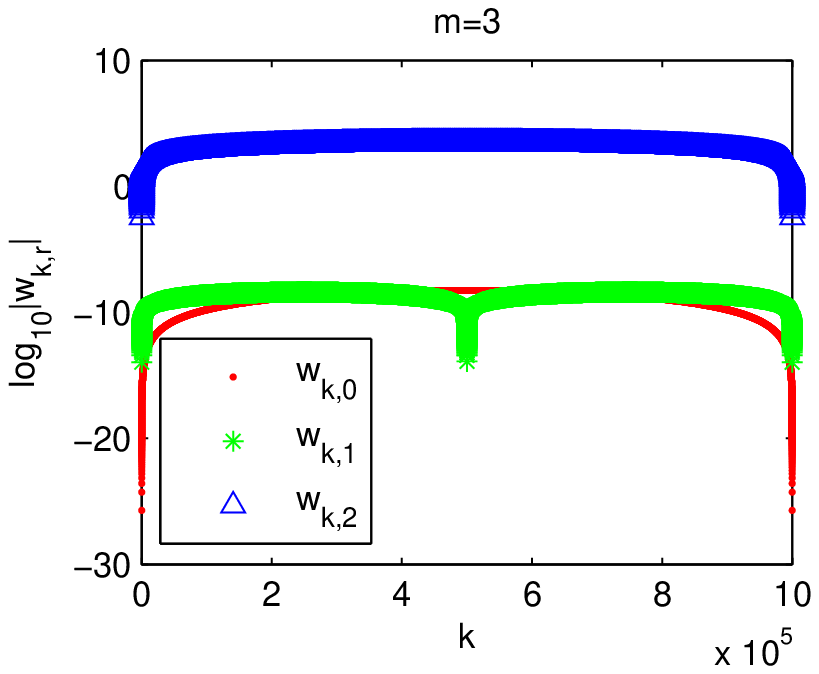}
  \includegraphics[width=1.8in]{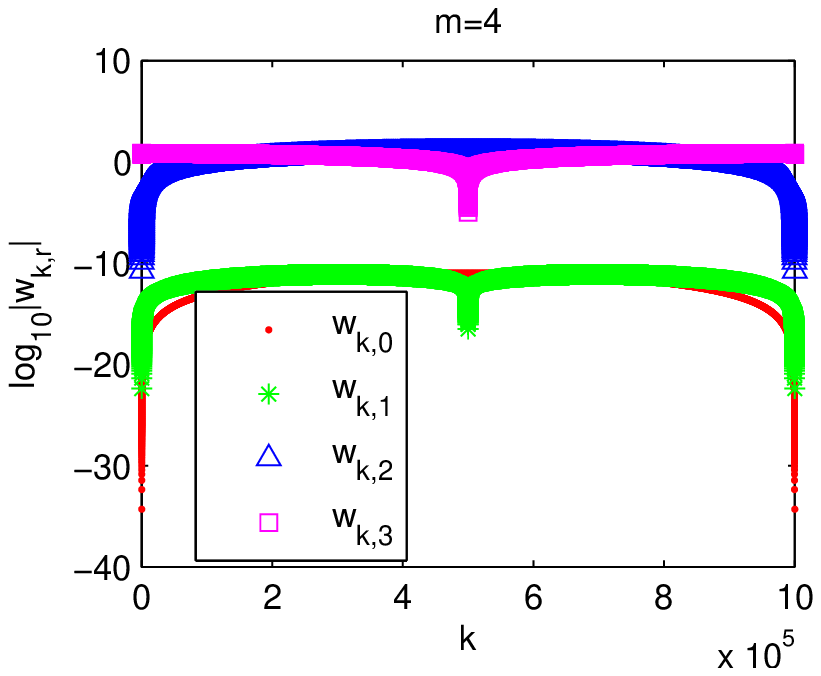}}
  \caption{Magnitude of the simplified barycentric weights $w_{k,r}$  by the $\mathbf{Algorithm~ 2}$ interpolating at the Jacobi-Gauss-Lobatto pointsystem  with $\alpha=\beta=1.5$: $k=1:10^6$ and $r=0:m-1$. }
\end{figure}

\begin{figure}[htbp]
\centerline{\includegraphics[width=1.8in]{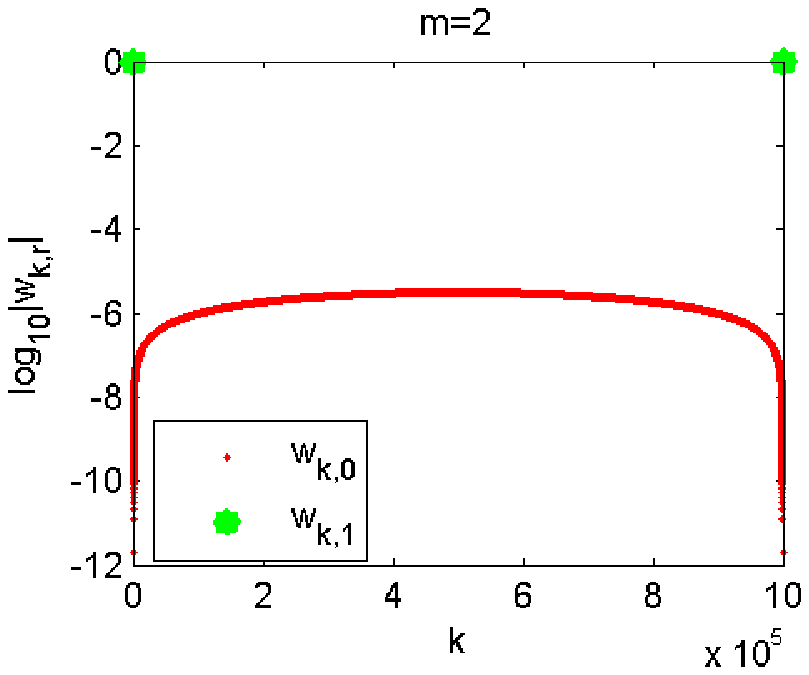} \includegraphics[width=1.8in]{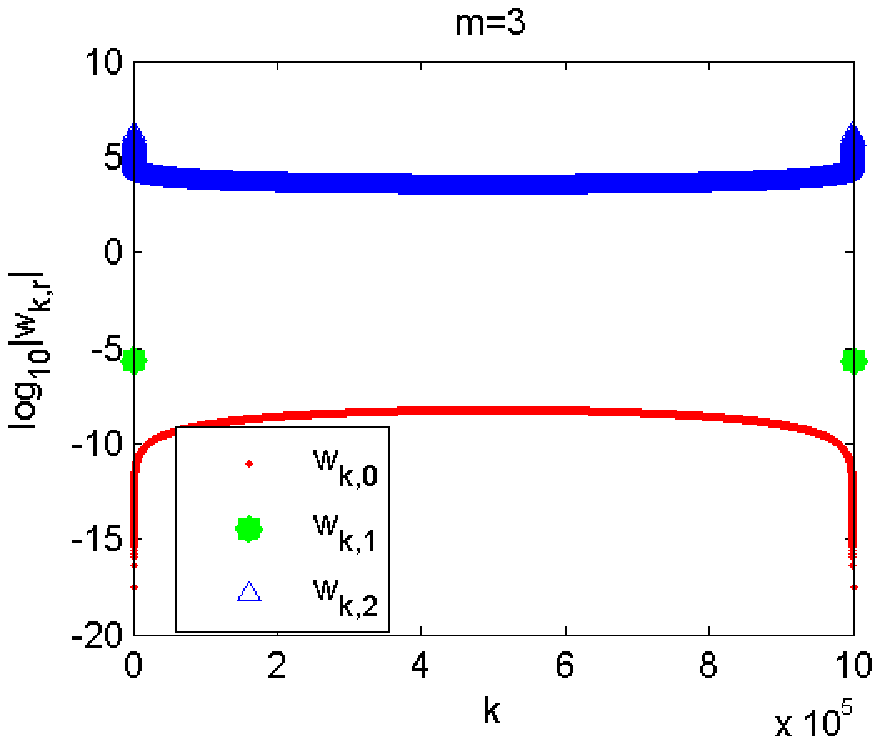}
  \includegraphics[height=1.42in,width=1.8in]{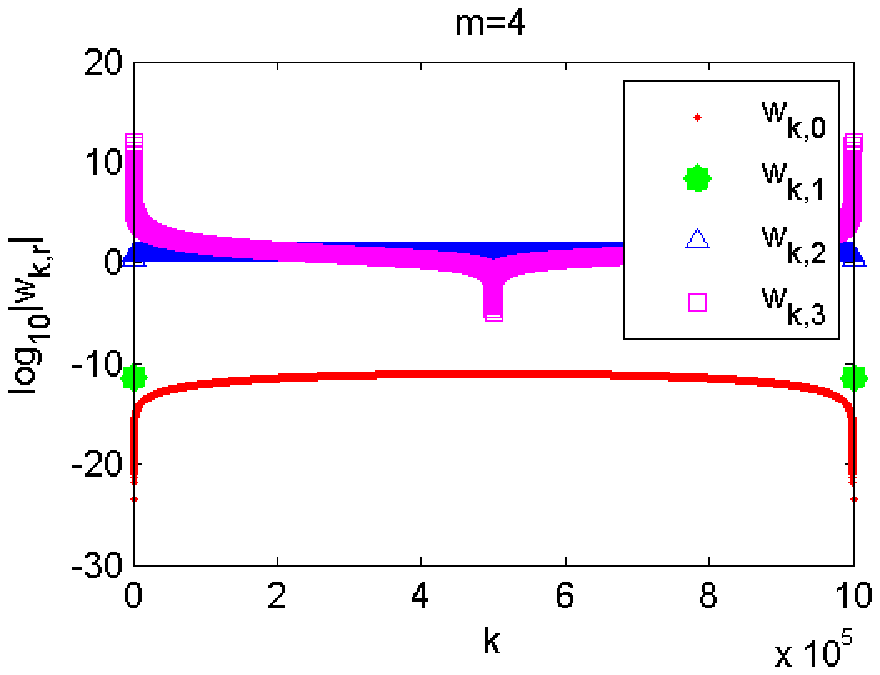}}
  \caption{Magnitude of the  simplified barycentric weights $w_{k,r}$  by the $\mathbf{Algorithm~ 2}$ interpolating at Legendre-Lobatto pointsystem: $k=1:10^6$ and $r=0:m-1$. $w_{k,1}=0$ except the first and the last term.}
\end{figure}

\begin{remark}
In the case $m=1$ and $\alpha=\beta=0.5$, the nodes $\{x_k\}_{k=1}^n$ are the Chebyshev points of second kinds
\begin{equation}
x_k=\cos\left(\frac{k-1}{n-1}\pi\right),\quad k=1,2,\ldots,n
\end{equation}
and $w_{k,0}=(-1)^{k+1}\eta_k$ with $\eta_k=\left\{\begin{array}{ll}\frac{1}{2},&k=1,n\\
 1,&{\rm otherwise}\end{array}\right.$ derived by Salzer \cite{Salzer}.
 \end{remark}

\begin{algorithm}[h]
\caption{Barycentric weights at Gauss-Jacobi or Jacobi-Gauss-Lobatto pointsystems}
\begin{algorithmic}[1]
\STATE Input parameters $n$, $m$, $\alpha$, $\beta$.
\STATE Compute the nodes $\{x_k\}_{k=1}^n$ and simplified barycentric weights $\{w_{k,0}\}_{k=1}^n$ by $\mathbf{jacpts}$.
\STATE Compute $M_{k,r}=\frac{\omega^{(r+1)}_n(x_k)}{(r+1)!\omega'_n(x_k)}$ $(k=1,\cdots,n,~r=1,\cdots,m-1)$ by recursion.
\STATE Compute $a_{k,i} = imM_{k,i}-\sum_{j=1}^{i-1}a_{k,j}M_{k,i-j}$ $(k=1,\cdots,n,~i=1,\cdots,m-1)$.
\STATE Compute $b_{k,i} = \frac{1}{i}\sum_{v=1}^{i}a_{k,v}b_{k,i-v}$ $(k=1,\cdots,n,~i=1,\cdots,m-1)$ with $b_{k,0}=1$.
\STATE Let $c_{k,0} = 1$, compute $c_{k,i} = -\sum_{j=1}^{i-1}c_{k,j}b_{k,i-j}$ $(k=1,\cdots,n,~i=1,\cdots,m-1)$.
\STATE Return $w_{k,i}=w_{k,0}c_{k,i}$ $(k=1,\cdots,n,~i=0,\cdots,m-1)$.
\end{algorithmic}
\end{algorithm}

\begin{algorithm}[h]
\caption{Barycentric weights at Gauss-Jacobi or Jacobi-Gauss-Lobatto pointsystems}
\begin{algorithmic}[1]
\STATE Input parameters $n$, $m$, $\alpha$, $\beta$.
\STATE Compute the nodes $\{x_k\}_{k=1}^n$ and simplified barycentric weights $\{w_{k,0}\}_{k=1}^n$ by $\mathbf{jacpts}$.
\STATE Compute $M_{k,r}=\frac{\omega^{(r+1)}_n(x_k)}{(r+1)!\omega'_n(x_k)}$ $(k=1,\cdots,n,~r=1,\cdots,m)$ by recursion.
\STATE Compute $\tilde{a}_{k,i} = -imM_{k,i}-\sum_{j=1}^{i-1}\tilde{a}_{k,j}M_{k,i-j}$ $(k=1,\cdots,n,~i=1,\cdots,m-1)$.
\STATE Compute $\tilde{b}_{k,i} = \frac{1}{i}\sum_{v=1}^{i}\tilde{a}_{k,v}\tilde{b}_{k,i-v}$ $(k=1,\cdots,n,~i=1,\cdots,m-1)$ with $\tilde{b}_{k,0}=1$.
\STATE Return $w_{k,i}=w_{k,0}\tilde{b}_{k,i}$ $(k=1,\cdots,n,~i=0,\cdots,m-1)$.
\end{algorithmic}
\end{algorithm}

\subsection{Lower order Hermite-Fej\'{e}r interpolation}
In particular, from (2.5), the barycentric weights for lower order Hermite-Fej\'{e}r barycentric interpolation can be given in the explicit forms.

\begin{itemize}
\item $m=2$: $ w_{k,1}=-\frac{\omega''_n(x_k)}{\omega'_n(x_k)}w_{k,0}$. Moreover, for the Gauss-Jacobi pointsystem,
$$ w_{k,1}=\left(\beta-\alpha-(\alpha+\beta+2)x_k\right)\overline{w}_{k},,\quad k=1,2,\ldots,n,
$$
while for the Jacobi-Gauss-Lobatto pointsystem,
$$w_{1,1}=\left(1+\frac{(n-2)(n+\alpha+\beta-1)}{\beta+1}\right)\beta\widehat{w}_1,\quad
w_{n,1}=\left(-1-\frac{(n-2)(n+\alpha+\beta-1)}{\alpha+1}\right)\alpha\widehat{w}_n$$
 and $$w_{k,1}=\left(\beta-\alpha-(\alpha+\beta-2)x_k\right)\frac{\widehat{w}_{k}}{1-x_k^2},\quad k=2,\ldots,n-1.$$

\item $m=3$: $$w_{k,1}=-\frac{3\omega''_n(x_k)}{2\omega'_n(x_k)}w_{k,0},\quad\quad w_{k,2}=\left(-\frac{\omega^{(3)}_n(x_k)}{2\omega'_n(x_k)}+\frac{3}{2}\left(\frac{\omega''_n(x_k)}{\omega'_n(x_k)}\right)^2\right)w_{k,0}.$$

\item $m=4$:
$$w_{k,1}=-\frac{2\omega''_n(x_k)}{\omega'_n(x_k)}w_{k,0},\quad\quad
w_{k,2}=\left(-\frac{2\omega^{(3)}_n(x_k)}{3\omega'_n(x_k)}+\frac{5}{2}\left(\frac{\omega''_n(x_k)}{\omega'_n(x_k)}\right)^2\right)w_{k,0}$$
and
$$w_{k,3}=\left(-\frac{\omega^{(4)}_n(x_k)}{6\omega'_n(x_k)}+\frac{5\omega''_n(x_k)\omega^{(3)}_n(x_k)}{3\omega'_n(x_k)\omega'_n(x_k)}
-\frac{5}{2}\left(\frac{\omega''_n(x_k)}{\omega'_n(x_k)}\right)^3\right)w_{k,0}.$$
\end{itemize}


\section{Illustration of numerical stability and numerical examples}

The stability for the second barycentric formulas for Lagrange interpolation has been extensively studied by Henrici \cite{Henricibook1982}, Berrut and Trefethen \cite{Berrut2004}. Rigorous arguments
that make this intuitive idea precise are provided by Higham \cite{Highambook,Higham2004}. 
For more details, see \cite{Berrut2004,Highambook,Higham2004}.

These arguments can be directly applied to the second barycentric higher order Hermite-Fej\'{e}r formula (2.15) if the barycentric weights can be evaluated well and do not suffer from overflows or underflows, which makes barycentric interpolation entirely reliable
in practice for small values of $m$ at the pointsystems, such as the Chebyshev pointsystem (1.3) and Jacobi-Gauss-Lobatto pointsystem for $\alpha=\beta=1.5$, i.e. the roots of $(1-x^2)P_{n-2}^{(\frac{3}{2},\frac{3}{2})}(x)$, studied in this paper.

To be pointed out especially, although both the $\mathbf{Algorithm~ 1}$ and $\mathbf{Algorithm~ 2}$ enjoy the fast implementation at the cost of $O(nm^2)$ operations for pointsystems discussed in Section 2, the performances are not always same. More specifically, they own the same high accuracy for small $m$ and different outcomes for larger $m$. The $\mathbf{Algorithm~ 2}$ manifests better stability for larger $m$ in our numerical experiments. Form the descriptions of the two algorithms, we can see that the first algorithm needs one more step than the second algorithm. This extra step leads to a great loss of significance due to the fast growth of the entries in these two algorithms for large $m$.  So, the $\mathbf{Algorithm~ 2}$ is recommended in practice.

Here, we use four functions to test the accuracy and stability for the second Hermite-Fej\'{e}r barycentric interpolation form  with $\mathbf{Algorithm~ 2}$ at the Gauss-Jacobi and Jacobi-Gauss-Lobatto pointsystems with $f(x) =1/(1+x^2)$, which is analytic
in a neighborhood of $[-1,1]$, $C^{\infty}$ function $f(x)=e^{-1/x^2}$, and nonsmooth functions $f(x)=1-|x|^3$.
Figures 3.1-3.3 illustrate the performance of the second barycentric interpolant (2.15) for the above first three
functions at Chebyshev pointsystem (1.3) and Jacobi-Gauss-Lobatto pointsystem  with $\alpha=\beta=1.5$
by using $n=10:20:2000$ grid points and
$m=2^{j-1}$ ($j=1,2,\ldots,6$) derivatives under the $\infty$-norm for the vector $\|f(x)-H_{mn-1}(f,x)\|_{\infty}$ at $x=-1:0.02:1$.

\begin{figure}[htbp]
\centerline{\includegraphics[height=4in,width=5.8in]{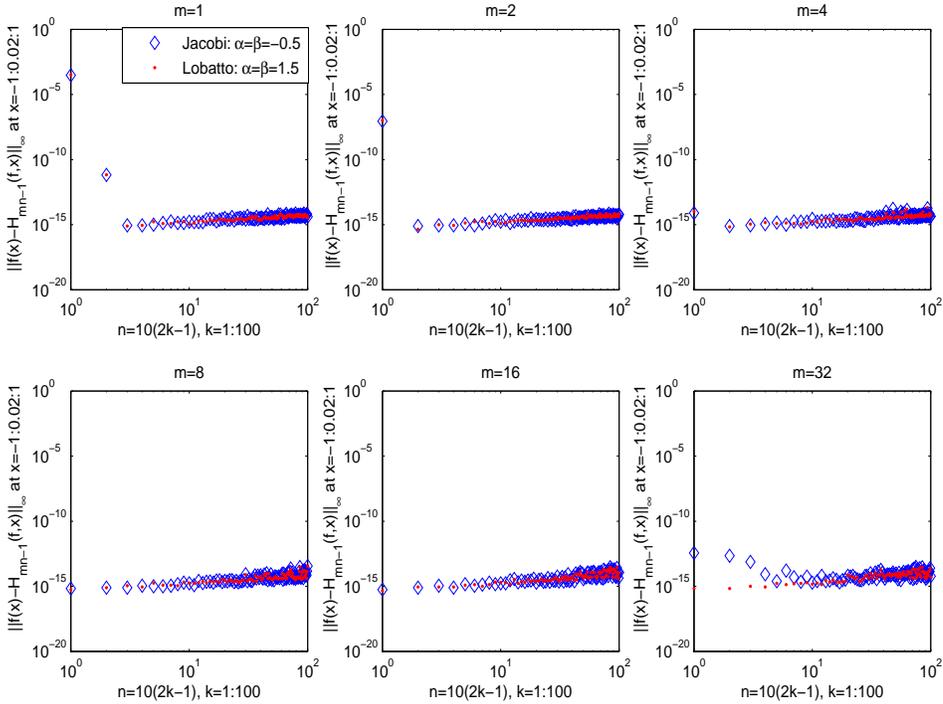} }
  \caption{$\|f(x)-H_{mn-1}(f,x)\|_{\infty}$ at $x=-1:0.02:1$ with $n=10:20:2000$ and $m=2^{j-1}$ ($j=1,2,\ldots,6$) for $f(x)=1/(1+x^2)$  at the Chebyshev pointsystem (1.3)  and Jacobi-Gauss-Lobatto pointsystem  with $\alpha=\beta=1.5$, respectively.}
\end{figure}

\begin{figure}[htbp]
\centerline{\includegraphics[height=4in,width=5.8in]{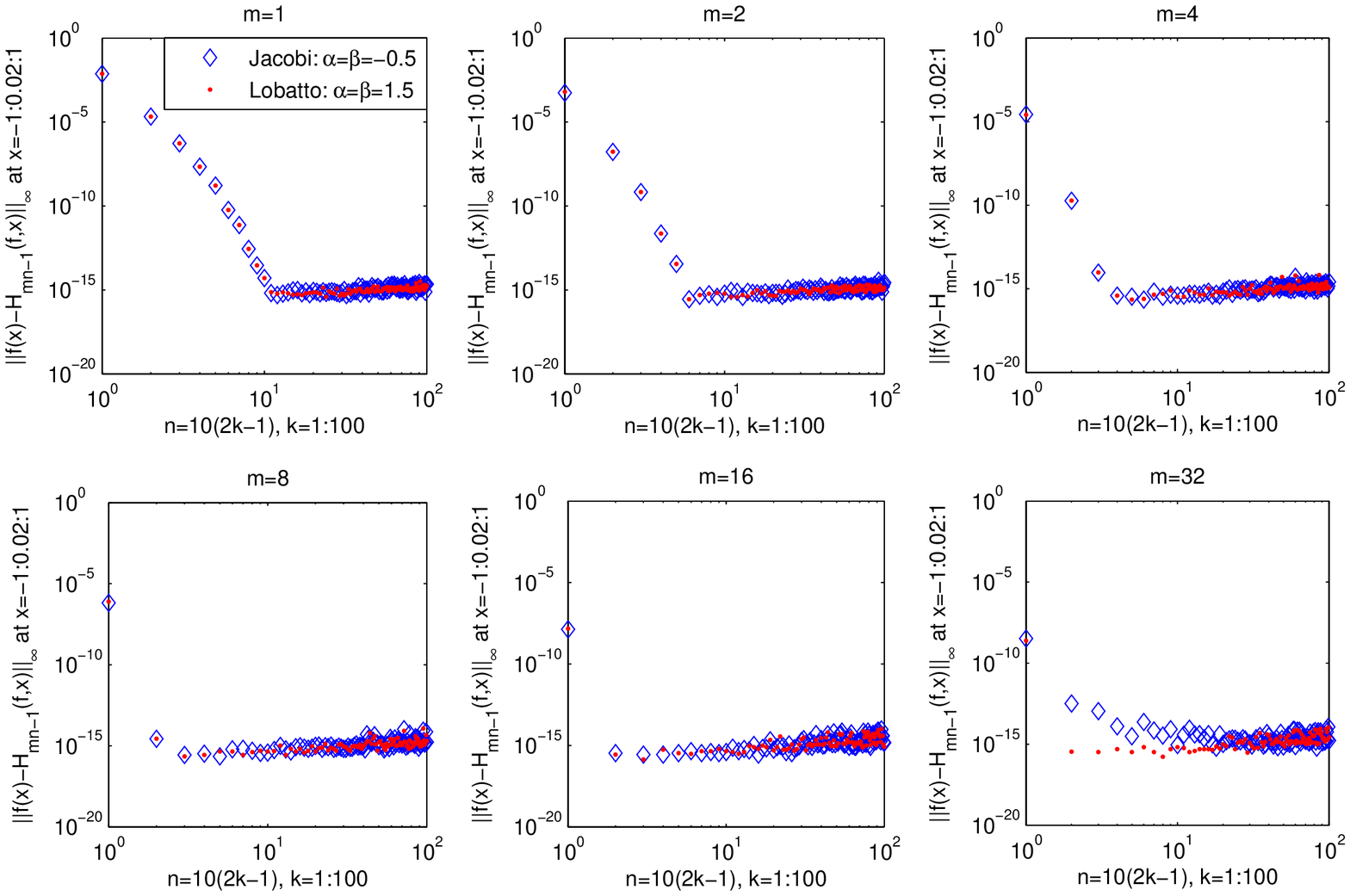} }
  \caption{$\|f(x)-H_{mn-1}(f,x)\|_{\infty}$ at $x=-1:0.02:1$ with $n=10:20:2000$ and $m=2^{j-1}$ ($j=1,2,\ldots,6$) for $f(x)=e^{-1/x^2}$  at the Chebyshev pointsystem (1.3)  and Jacobi-Gauss-Lobatto pointsystem  with $\alpha=\beta=1.5$, respectively.}
\end{figure}

\begin{figure}[htbp]
\centerline{\includegraphics[height=3.8in,width=5.8in]{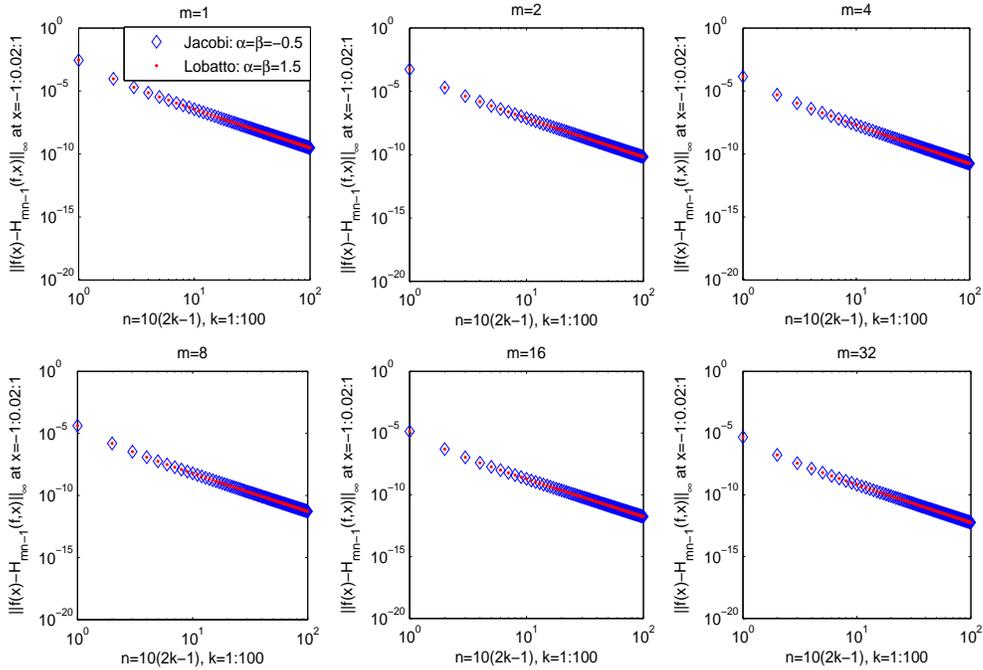} }
  \caption{$\|f(x)-H_{mn-1}(f,x)\|_{\infty}$ at $x=-1:0.02:1$ with $n=10:20:2000$ and $m=2^{j-1}$ ($j=1,2,\ldots,6$) for $f(x)=1-|x|^3$  at the Chebyshev pointsystem (1.3)  and Jacobi-Gauss-Lobatto pointsystem with $\alpha=\beta=1.5$, respectively.}
\end{figure}

%

From these examples, we can see that the barycentric interpolation is quite stable for small values of $m$. However, when $m$ is too large, the simplified barycentric weights will suffer from overflows or underflows too. Figure 3.5 shows the maximum $m$ for a fixed $n$ in the computation of the simplified barycentric weights $w_{k,r}$ $(k=1:n,r=0:m-1)$ by \textbf{Algorithm 2} before the overflows or underflows occurre.

\begin{figure}[htbp]
\centerline{\includegraphics[height=1.6in,width=2.7in]{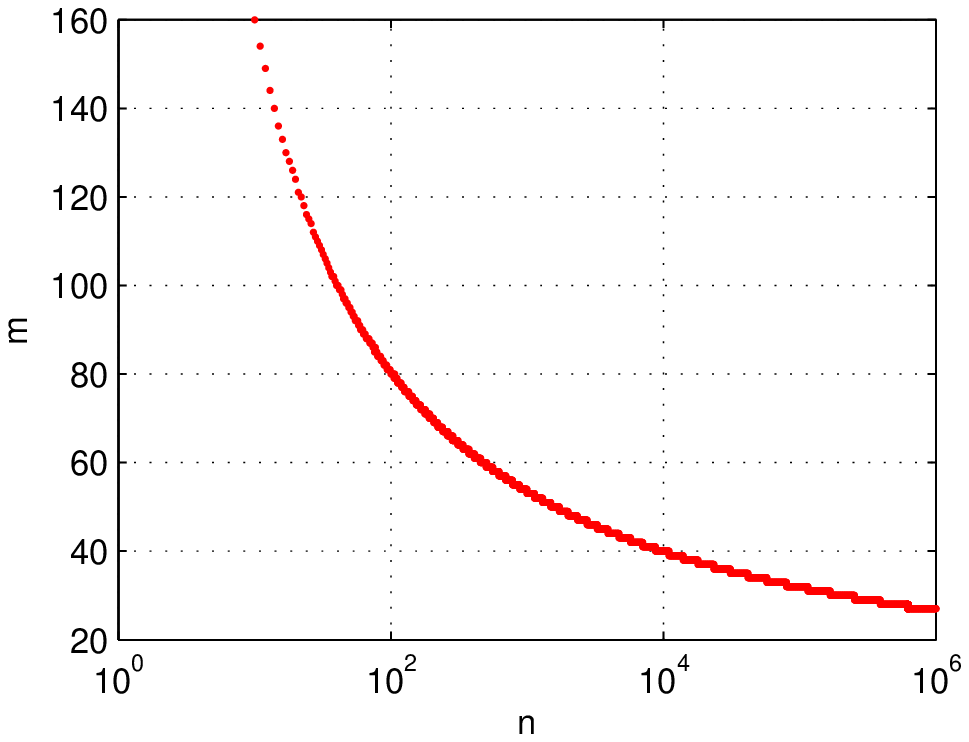} \includegraphics[height=1.6in,width=2.7in]{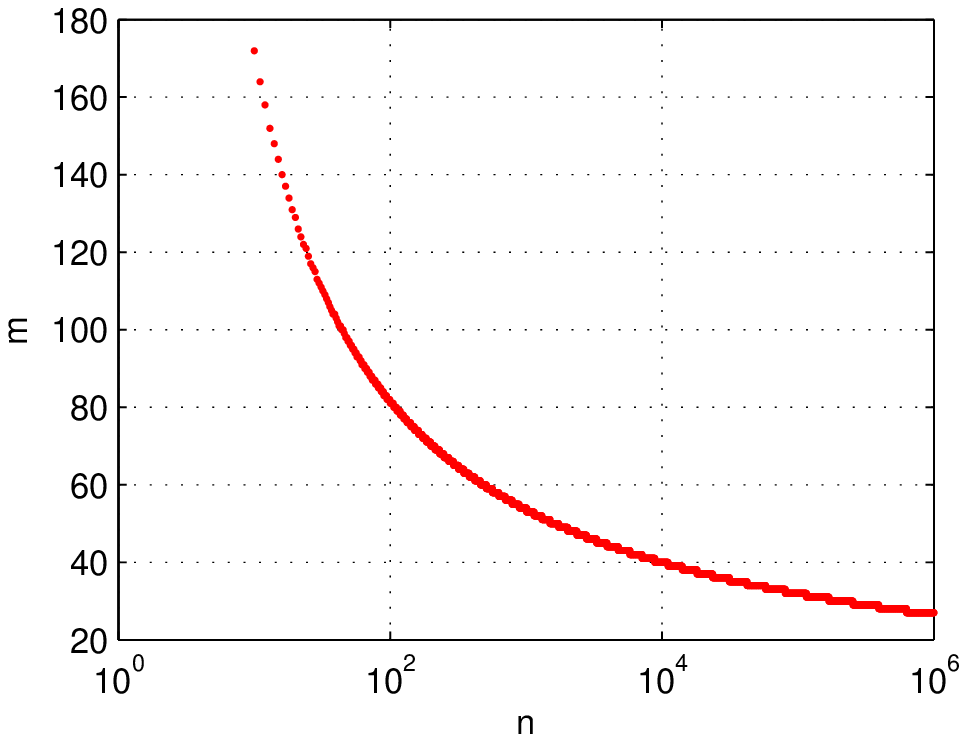}}
  \caption{The maximum number $m$ for a fixed $n$ before the overflows or underflows occurred for Gauss-Jacobi pointsystem $\alpha=\beta=-0.5$ (left) and for Jacobi-Gauss-Lobatto pointsystem $\alpha=\beta=1.5$ (right): $n=10:10^6$.}
\end{figure}

%

\section{Final remarks}
It is remarkable that Chebyshev pointsystems (1.3) and (2.32) are fairly nice in Lagrange polynomial approximation (see \cite{Trefethen1,Trefethen2,XiangChenWang}). However, for (higher order) Hermite-Fej\'{e}r interpolation, the Chebyshev pointsystem (2.32) completely fails (see Figure 4.1). The good choice is Chebyshev pointsystem (1.3) or the roots of $(1-x^2)P_{n-2}^{(\frac{3}{2},\frac{3}{2})}(x)$, since pointsystem (2.32) is not normal and the latter two pointsystems are strongly normal for $m=2$.

\begin{figure}[htbp]
\centerline{\includegraphics[height=3in,width=6in]{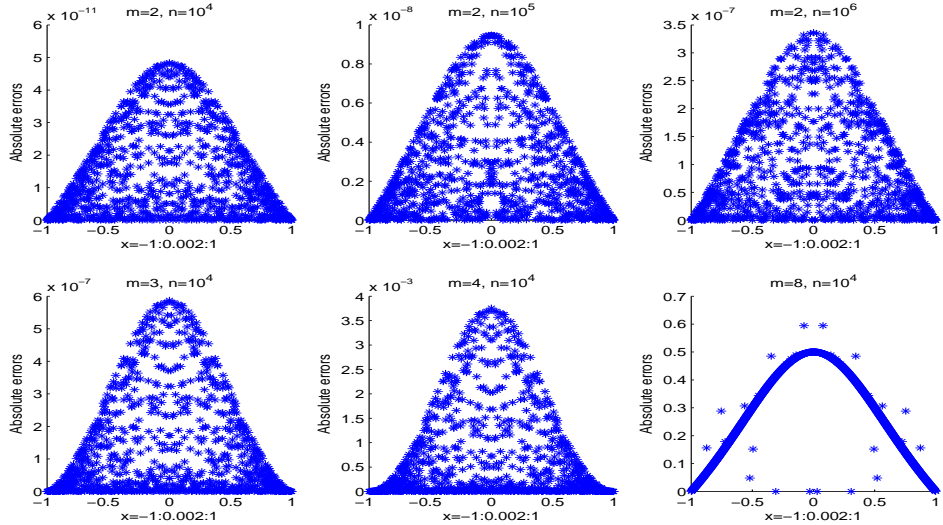}}
  \caption{The absolute errors of $H_{mn-1}(f,x)-f(x)$ at $x=-1:0.002:1$ with different $m$ and $n$ by using Chebyshev pointsystem (2.32) for $f(x)=\frac{1}{1+x^2}$.}
\end{figure}

For strongly normal pointsystem satisfying (1.6), V\'{e}rtesi \cite{Vertesi1979a} proved that for each $f\in C^1[-1,1]$,
$$
\|E(f)\|_{\infty}=\max_{x\in [-1,1]}|H_{2n-1}(f,x)-f(x)|\le \left(4+\frac{2}{c}\right)\min_{q_j\in {\cal P}_{2n-2}}\|f'-q_j\|_{\infty}
$$
where ${\cal P}_{2n-2}$ denotes the set of all polynomials of degree at most
$2n-1$ with real coefficients.

If $f$ is analytic or of finite limited regularity, the convergence rate on Hermite-Fej\'{e}r interpolation $H_{2n-1}(f,x)$ at Gauss-Jacobi pointsystem can be improved and given explicitely based on the asymptotics of the coefficients of Chebyshev series for $f$.

Suppose $f(x)$ satisfies a
Dini-Lipschitz condition on $[-1,1]$, then it has the following absolutely and
uniformly convergent Chebyshev series expansion (see Cheney \cite[pp 129]{Cheney})
\begin{equation}
f(x)=\sum_{j=0}^{\infty}{'}c_jT_j(x),\quad c_j=\frac{2}{\pi}\int_{-1}^{1}\frac{f(x)T_j(x)}{\sqrt{1-x^2}}dx,\quad
j=0,1,\ldots.
\end{equation}
where the prime denotes summation whose first term is halved,
$T_j(x)=\cos(j\cos^{-1}x)$ denotes the Chebyshev polynomial of
degree $j$.

\begin{lemma} (i)  (Bernstein \cite{Bernstein}) If $f$ is analytic with $|f(z)|\le
M$ in the region bounded by the ellipse ${\cal E}_{\rho}$ with foci $\pm1$ and major
and minor semiaxis lengths summing to $\rho>1$, then for each $j\ge
0$,
\begin{equation}
|c_j|\le {\displaystyle\frac{2M}{\rho^j}}.
\end{equation}

(ii) (Trefethen \cite{Trefethen1,Trefethen2}) For an integer $k\ge 1$, if $f(x)$ has an
absolutely continuous $(k-1)$st derivative $f^{(k-1)}$ on $[-1,1]$
and a $k$th derivative $f^{(k)}$ of bounded variation
$V_k={\rm Var}(f^{(k)})<\infty$, then for each $j\ge k+1$,
\begin{equation}
|c_j|\le{\displaystyle\frac{2V_k}{\pi j(j-1)\cdots(j-k)}}.
\end{equation}
\end{lemma}

\begin{lemma}
Suppose $\{x_j\}_{j=1}^n$ are the roots of $P_n^{(\alpha,\beta)}(x)$ ($\alpha,\beta>-1$), then it follows
\begin{equation}
\quad  (x-x_j)\ell_j(x)=\sigma_n\frac{\sqrt{(1-x_j^2)\overline{w}_j}}{2^{(\alpha+\beta+1)/2}}\sqrt{\frac{n!\Gamma(n+\alpha+\beta+1)}
{\Gamma(n+\alpha+1)\Gamma(n+\beta+1)}}P_n^{(\alpha,\beta)}(x),j=1,2,\ldots,n.
\end{equation}
\end{lemma}
\begin{proof}
Let $z_n=\int_{-1}^1(1-x)^{\alpha}(1+x)^{\beta}[P_n^{(\alpha,\beta)}(x)]^2dx$ and $K_n$ the leading coefficient of $P_n^{(\alpha,\beta)}(x)$. From Abramowitz and Stegun \cite{Abram}, we have
$$
z_n=\frac{2^{\alpha+\beta+1}}{2n+\alpha+\beta+1}\cdot\frac{\Gamma(n+\alpha+1)\Gamma(n+\beta+1)}{n!\Gamma(n+\alpha+\beta+1)},\quad K_n=\frac{1}{2^{n}}\frac{\Gamma(2n+\alpha+\beta+1)}{n!\Gamma(n+\alpha+\beta+1)}.
$$
Furthermore, by Hale and Townsend \cite{HaleTownsend} and Wang et al. \cite{Wang2013}, we obtain
$$\begin{array}{lll}
(x-x_j)\ell_j(x)=\frac{w_n(x)}{w_n'(x_j)}&=&\sigma_n(-1)^j\sqrt{\frac{K_n^22n(1-x_j^2)\overline{w}_j}{2n(2n+\alpha+\beta+1)z_n}}w_n(x)\\
&=&\sigma_n(-1)^j\sqrt{\frac{(1-x_j^2)\overline{w}_j}{(2n+\alpha+\beta+1)z_n}}P_n^{(\alpha,\beta)}(x),\end{array}
$$
which leads to the desired result (4.4).
\end{proof}

\begin{theorem}
Suppose $\{x_j\}_{j=1}^n$ are the roots of $P_n^{(\alpha,\beta)}(x)$ ($-1<\alpha,\beta\le 0$), then the Hermite-Fej\'{e}r interpolation (1.5) at  $\{x_j\}_{j=1}^n$ has the convergence rate
\begin{equation}\|E(f)\|_{\infty}\le\left\{\begin{array}{ll} {\displaystyle\frac{4\tau_nM[2n\rho^2+(1-2n)\rho]}{(\rho-1)^2\rho^{2n}}}\quad(n\ge 1),&\mbox{if $f$ analytic in ${\cal E}_{\rho}$ with $|f(z)|\le M$}\\
{\displaystyle\frac{4\tau_nV_k}{(k-1)\pi
(2n-1)(2n-2)\cdots(2n-k+1)}},&\mbox{if $f,\ldots,f^{(k-1)}$ absolutely continuous }\\
&\mbox{ and $V_k<\infty$, $n\ge k/2$, $k\ge 2$},\end{array}\right.
\end{equation}
where $E(f,x)=f(x)-H_{2n-1}(f,x)$, and
\begin{equation}
\tau_n=\left\{\begin{array}{ll}
O(n^{-1.5-\min\{\alpha,\beta\}}\log n),& \mbox{if $-1<\min\{\alpha,\beta\}\le\max\{\alpha,\beta\}\le -\frac{1}{2}$}\\
O(n^{2\max\{\alpha,\beta\}-\min\{\alpha,\beta\}-\frac{1}{2}}),& \mbox{if $-1<\min\{\alpha,\beta\}\le -\frac{1}{2}<\max\{\alpha,\beta\}\le 0$}\\
O(n^{2\max\{\alpha,\beta\}}),& \mbox{if $-\frac{1}{2}<\min\{\alpha,\beta\}\le \max\{\alpha,\beta\}\le 0$}\end{array}.\right.
\end{equation}
\end{theorem}
\begin{proof}
Since the Chebyshev series expansion of $f(x)$ is uniformly convergent under the assumptions of Theorem 4.3, and the error of Hermite-Fej\'{e}r interpolation (1.5) on Chebyshev polynomials satisfies $|E(T_{j},x)|=|T_j(x)-H_{2n-1}(T_{j},x)|=0$ for $j=0,1,\ldots, 2n-1$, then it yields
\begin{equation}
|E(f,x)|=|f(x)-H_{2n-1}(f,x)|=|\sum_{j=0}^{\infty}|c_j||E(T_{ j},x)|\le \sum_{j=2n}^{\infty}|c_j||E(T_{j},x)|.
\end{equation}
Furthermore, $|E(T_{j},x)|=|T_j(x)-\sum_{i=1}^nT_j(x_i)h_i(x)-\sum_{i=1}^nT_j'(x_i)b_i(x)|$. In the following, we will fucus on estimates on $|E(T_{ j},x)|$ for $j\ge 2n$.

Notice that the pointsystem  is normal  which implies $h_i(x)\ge 0$ for all $i=1,2,\ldots,n$ and $x\in [-1,1]$,
$$
1\equiv\sum_{i=1}^nh_i(x)=\sum_{i=1}^nv_i(x)\ell^2_i(x)
$$
(see \cite{Fejer1932a}) and then
\begin{equation}
|\sum_{i=1}^nT_j(x_i)h_i(x)|\le \sum_{i=1}^nh_i(x)=1,\quad j=0,1,\ldots.
\end{equation}
Additionally, by Lemma 4.2, it obtains for $j=2n,2n+1,\ldots$ that
$$\begin{array}{lll}
&&|\sum_{i=1}^nT_j'(x_i)b_i(x)|\\&=& j|\sum_{i=1}^nU_{j-1}(x_i)(x-x_i)\ell_i^2(x)|\\
&=&\frac{j}{2^{(\alpha+\beta+1)/2}}\sqrt{\frac{n!\Gamma(n+\alpha+\beta+1)}
{\Gamma(n+\alpha+1)\Gamma(n+\beta+1)}}|P_n^{(\alpha,\beta)}(x)\sum_{i=1}^nU_{j-1}(x_i)\sqrt{(1-x_i^2)\overline{w}_i}\ell_i(x)|\\
&=&\frac{j}{2^{(\alpha+\beta+1)/2}}\sqrt{\frac{n!\Gamma(n+\alpha+\beta+1)}
{\Gamma(n+\alpha+1)\Gamma(n+\beta+1)}}|P_n^{(\alpha,\beta)}(x)\sum_{i=1}^n\sin((j-1)\arccos(x_i))\sqrt{\overline{w}_i}\ell_i(x)|\\
&=&jO\left(|P_n^{(\alpha,\beta)}(x)|\sqrt{\|\{\overline{w}_i\}_{i=1}^n\|_{\infty}}\Lambda_n\right)\end{array}
$$
since $\sqrt{\frac{n!\Gamma(n+\alpha+\beta+1)}
{\Gamma(n+\alpha+1)\Gamma(n+\beta+1)}}$ is decreasing as $n$ increases and then uniformly bounded on $n$ for $-1<\alpha,\beta\le 0$, where $\Lambda_n=\max_{x\in [-1,1]}\sum_{i=1}^n|\ell_i(x)|$ is the Lebesgue constant, which, together with
$${\small
P_n^{(\alpha,\beta)}(x)=\left\{\begin{array}{ll}
O(n^{-\frac{1}{2}}),&\mbox{if $\max\{\alpha,\beta\}\le -\frac{1}{2}$}\\
O(n^{\max\{\alpha,\beta\}}),&\mbox{if $\max\{\alpha,\beta\}> -\frac{1}{2}$}
\end{array}\right.,
\overline{w}_i=\left\{\begin{array}{ll}
O(n^{-2-2\min\{\alpha,\beta\}}),&\mbox{if $\min\{\alpha,\beta\}\le -\frac{1}{2}$}\\
O(n^{-1}),&\mbox{if $\min\{\alpha,\beta\}> -\frac{1}{2}$}
\end{array}\right.
}$$
(see Szeg\"{o} \cite[pp 168, 354]{Szego}) and
$$
\Lambda_n=\left\{\begin{array}{ll}
O(\log n ),& \mbox{if $\max\{\alpha,\beta\}\le -\frac{1}{2}$}\\
O(n^{\max\{\alpha,\beta\}+\frac{1}{2}}),& \mbox{if $\max\{\alpha,\beta\}> -\frac{1}{2}$}\end{array}\right.\mbox{\quad (\cite[pp 338]{Szego})},
$$
yields
\begin{equation}
|\sum_{i=1}^nT_j'(x_i)b_i(x)|=j\tau_n.
\end{equation}

Thus, by (4.8), (4.9) and (1.5), we find $|E(T_{ j},x)|\le 2+j\tau_n$ for $j\ge 2n$, and then the error of Hermite-Fej\'{e}r interpolation (4.7)  satisfies
$$
|E(f,x)|=|f(x)-H_{2n-1}(f,x)|\le \sum_{j=2n}^{\infty}|c_j||E(T_{ j},x)|=2\tau_n\sum_{j=2n}^{\infty}j|c_j|,
$$
which, following \cite{XiangChenWang}, leads to the desired result.
\end{proof}

From the definition of $\tau_n$ (4.6), we see that when $\alpha=\beta=-\frac{1}{2}$ the convergence order on $n$ is the lowest. In addition, from Szabados \cite{Sza1993} (also see Sadiq and Viswanath \cite{Sadiq2013}), we see that the convergence of the higher order Hermite-Fej\'{e}r interpolation (2.15) at the Chebyshev pointsystem (1.3) satisfies
\begin{equation}
\|f-H_{mn-1}(f)\|_{\infty}=\left\{\begin{array}{ll}O(\log n)\|f-p*\|_{C^{m-1}[-1,1]},&\mbox{if $m$ is odd}\\
O(1)\|f-p*\|_{C^{m-1}[-1,1]},&\mbox{if $m$ is even}\end{array}\right.
\end{equation}
where $p*$ is the best approximation polynomial of $f$ with degree at most $mn-1$ and $\|f-p*\|_{C^{m-1}[-1,1]}=\max_{0\le j\le m-1}\|f^{(j)}-(p^*)^{(j)}\|_{\infty}$.

Numerical examples also illustrate that the roots of $(1-x^2)P_{n-2}^{(\frac{3}{2},\frac{3}{2})}(x)$ are appropriate to higher order Hermite-Fej\'{e}r interpolation. In the future work, we will consider the convergence rates on this pointsystem.

It is worth noting that the new methods for Hermite barycentric weights at Gauss-Jacobi pointsystems or Jacobi-Gauss-Lobatto pointsystems can be extended to Jacobi-Gauss-Radau pointsystems or the roots of other kinds of orthogonal polynomials, such as Laguerre polynomials, Hermite polynomials, etc., based on the works of \cite{Glaser}, \cite{Wang2013} and {\sc Chebfun} \cite{Chebfun}.


\end{document}